\documentclass[11pt,a4paper]{amsart}
\usepackage{amsfonts,amsmath,amssymb}
\usepackage{hyperref}
\usepackage{latexsym,fullpage,amsfonts,amssymb,amsmath,amscd,graphics,epic}
\usepackage[all]{xy}
\usepackage{amssymb,amsthm,amsxtra}
\usepackage[usenames]{color}
\usepackage{amscd}
\usepackage{amsthm}
\usepackage{amsfonts}
\usepackage{amssymb}
\usepackage{mathrsfs}
\usepackage{url}
\usepackage{bbm}
\usepackage{wasysym}
\usepackage{enumitem}
\usepackage[vcentermath]{youngtab}%
\theoremstyle{plain}
\newtheorem*{theorem*}{Theorem}
\newtheorem*{remark*}{Remark}
\newtheorem*{example*}{Example}
\newtheorem{lemma}{Lemma}[subsection]
\newtheorem{proposition}[lemma]{Proposition}
\newtheorem{corollary}[lemma]{Corollary}
\newtheorem{theorem}[lemma]{Theorem}

\newtheorem*{conjecture*}{Conjecture}

\newtheorem{introtheorem}{Theorem}

\theoremstyle{definition}
\newtheorem{definition}[lemma]{Definition}

\newtheorem{example}[lemma]{Example}

\theoremstyle{remark}
\newtheorem{remark}[lemma]{Remark}
\newtheorem{construction}[lemma]{Construction}

\newtheorem{notation}[lemma]{Notation}

\newcommand{\Hom}{\operatorname{Hom}}

\newcommand{\eps}{\varepsilon}
\newcommand{\triv}{{\mathbbm 1}}

\newcommand{\Ind}{\operatorname{Ind}}
\newcommand{\id}{\operatorname{Id}}

\newcommand{\Ker}{\operatorname{Ker}}

\newcommand{\End}{\operatorname{End}}

\newcommand{\bC}{{\mathbb C}}
\newcommand{\bZ}{{\mathbb Z}}

\newcommand{\lam}{{\lambda}}

\newcommand{\gl}{{\mathfrak{gl}}}
\newcommand{\sll}{{\mathfrak{sl}}}
\newcommand{\abs}[1]{\left|{#1}\right|}
\newcommand{\C}{{\mathcal{C}}}
\newcommand{\cV}{{\mathcal{V}_t}}
\newcommand{\cD}{{\mathcal{D}}}
\newcommand{\rx}{{\mathrm{x}}}

\newcommand{\Repc}{{Rep_{_{\C}}}}
\newcommand{\uRep}{{\underline{Rep}}}

\newcommand{\InnaA}[1]{{{#1}}}

\newcommand{\InnaC}[1]{{#1}}
\newcommand{\InnaD}[1]{{#1}}
\def\quotient#1#2{%
    \raise1ex\hbox{$#1$}\Big/\lower1ex\hbox{$#2$}%
}

\begin{document}

\date{\today}
\title{Categorical actions and multiplicities in the Deligne category $\underline{Rep}(GL_t)$}
 \author{Inna Entova-Aizenbud}
\address{Inna Entova-Aizenbud,
Mathematics department, Ben Gurion University, Be'er-Sheva,
Israel}
\email{inna.entova@gmail.com}

\begin{abstract}
 We study the categorical type A action on the Deligne category $\mathcal{D}_t=\uRep(GL_t)$ ($t \in \bC$) and its ``abelian envelope'' $\cV$ constructed in \cite{EHS}. 
 
 For $t \in \bZ$, this action categorifies an action of the Lie algebra $\sll_{\bZ}$ on the tensor product of the Fock space $\mathfrak{F}$ with $\mathfrak{F}_t^{\vee}$, its restricted dual ``shifted'' by $t$, as was suggested by I. Losev. In fact, this action makes the category $\cV$ the tensor product (in the sense of Losev and Webster, \cite{LW}) of categorical $\sll_{\bZ}$-modules $Pol$ and $Pol_t^{\vee}$. The latter categorify $\mathfrak{F}$ and $\mathfrak{F}_t^{\vee}$ respectively, the underlying category in both cases being the category of stable polynomial representations (also known as the category of Schur functors), as described in \cite{HY, L}.
 
 When $t \notin \bZ$, the Deligne category $\mathcal{D}_t$ is abelian semisimple, and the type A action induces a categorical action of $\sll_{\bZ} \times \sll_{\bZ}$. This action categorifies the $\sll_{\bZ} \times \sll_{\bZ}$-module $\mathfrak{F} \boxtimes \mathfrak{F}^{\vee}$, making $\mathcal{D}_t$ the exterior tensor product of the categorical $\sll_{\bZ}$-modules $Pol$, $Pol^{\vee}$.
 
 Along the way we establish a new relation between the Kazhdan-Lusztig coefficients  and the multiplicities in the standard filtrations of tilting objects in $\cV$.
\end{abstract}

\maketitle
\setcounter{tocdepth}{3}
\section{Introduction}
In this paper we study the categorical type A action on the Deligne categories $\uRep(GL_t)$, and the induced categorical $\sll_{\bZ}$-action on their abelian versions $\cV$ defined in \cite{EHS} for $t \in \bZ$.
\subsection{Categorical actions}
The theory of categorical type A actions was first introduced by Chuang and Rouquier in \cite{CR}, and further developed by Rouquier (\cite{R}), Khovanov and Lauda (see for example \cite{KL}), Brundan, Losev, Webster and others (see \cite{BLW,HY,LW}, and the review \cite{L}). 

According to this theory, a categorical type A action on an abelian category $\mathcal{A}$ is an adjoint pair of (exact) endofunctors $E, F$ of $\mathcal{A}$, and natural transformations $\tau \in \End(F^2)$, $\rx \in \End(F)$, defining an action of the degenerate affine Hecke algebra $dAHA_d$ on $F^d$ for any $d$. 

If the generalized eigenvalues of $\rx \in \End(F)$ are integers, the type A action induces an $\sll_{\bZ}$-action. It is given by the decompositions $F = \sum_{a \in \bZ} F_a$, $E = \sum_{a \in \bZ} E_a$, corresponding to the generalized eigenvalues $a$ of $\rx \in \End(F)$ and the induced natural transformation $\rx^{\vee} \in \End(E)$. The obtained functors $E_a, F_a$ are once again an adjoint pair; furthermore, we require that $E_a$ be isomorphic to the right adjoint of $F_a$. On the level of the Grothendieck group of $\mathcal{A}$, the functors $E_a, F_a$ induce an action of $\sll_{\bZ}$ on $Gr(\mathcal{A})$ through its generators $e_a, f_a$ ($a \in \bZ$). The natural transformation $\tau$ serves to categorify the relations $[e_a, f_b] = 0$ for $a \neq b$, and $[e_a, f_a]=h_a$. See \cite{CR, R, BLW} for explicit definitions and details.

Several important examples of categorical $\sll_{\bZ}$-actions have been found; in particular, it was shown that the category $Rep(\gl(m|n))$ of finite-dimensional representations of the Lie superalgebra $\gl(m|n)$ carries a natural action of $\sll_{\bZ}$, where $F, E$ are given by tensoring with the standard representation $\bC^{m|n}$ and its dual. 

Then the (complexified) Grothendieck group of $Rep(\gl(m|n))$ corresponds to the representation $\wedge^m \bC^{\bZ} \otimes \wedge^n (\bC^{\bZ})^*$ of $\sll_{\bZ}$.

\subsection{Deligne categories}
In \cite{DM}, Deligne and Milne constructed a family of rigid symmetric monoidal categories $\uRep(GL_t)$ (denoted by $\mathcal{D}_t$ in this paper), parametrized by $t \in \mathbb{C}$. These categories satisfy several properties: 
\begin{itemize}
\item The category $\mathcal{D}_t$ is the universal Karoubian additive symmetric monoidal category generated by a dualizable object of dimension $t$. We will denote this object by $V_t$.
 \item For $t = n \in \bZ_+$, this universal property provides a symmetric monoidal functor $\mathbf{S}_{n}:\mathcal{D}_{t=n} \longrightarrow Rep(GL(n))$, $V_t \mapsto \bC^n$ ($Rep(GL(n))$ is the category of finite-dimensional representations of $GL(n)$). This functor is full and essentially surjective; due to this fact, one can consider $\mathcal{D}_t$ as a polynomial family ``interpolating'' the categories $Rep(GL(n))$.                                                                                                                                                                                         \item For $t$ non-integer, these are semisimple tensor categories. 
                                                                                                                                                                                                 \end{itemize}
 
 The above universal property also provides symmetric monoidal functors $\mathbf{S}_{m, n}:\mathcal{D}_{t} \longrightarrow Rep(GL(m|n))$, $V_t \mapsto \bC^{m|n}$ whenever $t \in \bZ$, $m, n \in \bZ_+$, $m-n = t$ (here $Rep(GL(m|n))$ is the category of finite-dimensional representations of the algebraic supergroup $GL(m|n)$).
 
 For $t\in\bZ$ the category $\mathcal{D}_t$ is Karoubian but not abelian. To remedy this, Hinich, Serganova and the author constructed in \cite{EHS} a tensor category $\cV$ which satisfies a universal property: given a tensor category $\mathcal{C}$, the exact tensor functors $\mathcal{V}_t \rightarrow \mathcal{C}$ classify the $t$-dimensional objects in $\mathcal{C}$ not annihilated by any Schur functor (if $t \notin \bZ$, this condition is vacuous). 
 
 The category $\mathcal{D}_t$ embeds into $\cV$ for any $t$ as a full rigid symmetric monoidal subcategory. When $t \notin \bZ$, the category $\cV$ is defined to be just $\mathcal{D}_t$.

 In addition to the properties above, the construction of the category $\cV$ demonstrates a certain stabilization phenomenon in the categories of representations of algebraic supergroups $GL(m|n)$, and $\cV$ can be seen as a ``stable'' inverse limit of $Rep(GL(m|n))$ for $m-n =t$ as $m, n \to \infty$.

 In the category $\cV$ we distinguish three sets of isomorphism classes of 
objects, all three parametrized by the set of bipartitions of arbitrary size 
(denoted by $\lam = (\lam^{\bullet}, \lam^{\circ})$). The three sets are: simple 
objects $\mathbf{L}(\lam)$, standard\footnote{The terminology comes from the 
theory of highest-weight categories. The category $\cV$ is not a highest-weight 
category (it does not have projectives nor injectives), but it is 
``\InnaC{lower} highest-weight'', meaning that it has a filtration by full 
subcategories which are highest-weight. The objects $\mathbf{V}(\lam), 
\mathbf{T}(\lam)$ lying in these subcategories are respectively standard and 
tilting objects.} objects $\mathbf{V}(\lam)$ and tilting objects 
$\mathbf{T}(\lam)$.
 
 When $t \notin \bZ$, we have: $\mathbf{L}(\lam) = \mathbf{V}(\lam)= \mathbf{T}(\lam)$.
 
 \subsection{} The categorical type A action on $\cD_t$ is given by functors $F, E \in \End(\mathcal{D}_t)$, $F := V_t \otimes (\cdot)$, $E :=V_t^* \otimes (\cdot)$; the natural transformations $\tau ,\rx $ are described in Construction \ref{constr:skeletal_action_Rep_gl}.  One should beware that the category $\mathcal{D}_t$ is not abelian, and the eigenvalues of the operator $\rx$ are not necessarily integral, making it a generalized type A action in the sense of \cite[Section 5.1.1]{R}.

 This action extends to a generalized type A action on the abelian category $\cV$, and the functors $\mathbf{S}_{m, n}: \mathcal{D}_{t=m-n} \to Rep(GL(m|n))$, $V_t \to \bC^{m|n}$ respect the type A actions on these categories.
%
%
 
\subsection{Results}
Our first main result is the description of the categorical actions on $\mathcal{D}_t$ and $\cV$.

Consider the Fock space representation $\rho_{\mathfrak{F}}:\sll_{\bZ} \longrightarrow \End(\mathfrak{F})$, with a basis $\{v_{\nu}\}$ of weight vectors parametrized by partitions of arbitrary size. We will also consider the ``twisted'' dual Fock representation $\rho_{\mathfrak{F}}^{\vee}:\sll_{\bZ} \longrightarrow \End(\mathfrak{F})$ (denoted $\mathfrak{F}^{\vee}$ for short), where $f_a, e_a$ act by $\rho_{\mathfrak{F}}(f_{-a})^T, \rho_{\mathfrak{F}}(e_{-a})^T$.

Next, \InnaA{for $t \in \bZ$}, consider the \InnaA{representation $\mathfrak{F}^{\vee}_t$ of $ \sll_{\bZ}$: the underlying vector space is $\mathfrak{F}^{\vee}$, with $f_a, e_a$ ($a \in \bZ$) acting by $\rho_{\mathfrak{F}}^{\vee}(f_{a+t}), \rho_{\mathfrak{F}}^{\vee}(e_{a+t})$.}


The following statements were suggested by I. Losev:
\begin{introtheorem}\label{introthrm:action_Deligne_cat}
\mbox{}
\begin{enumerate}
 \item Let $t \notin \bZ$ (so $\mathcal{D}_t$ is semisimple).

 The functors $F = V_t \otimes(\cdot)$, $E =V_t^* \otimes(\cdot) $ induce an action of the Lie algebra $\sll_{\bZ}\times \InnaA{\sll_{\bZ}}$ on the complexified Grothendieck group $\bC \otimes_{\bZ} Gr(\mathcal{D}_t)$. 
 
 As a $\sll_{\bZ}\times \InnaA{\sll_{\bZ}}$-module, $$\bC \otimes_{\bZ}Gr(\mathcal{D}_t) \stackrel{\sim{}}{\longrightarrow} \mathfrak{F} \boxtimes \mathfrak{F}^{\vee}, \; \; \left[ \mathbf{T}(\lam) \right] \mapsto v_{\lam^{\bullet}} \otimes v_{\lam^{\circ}}$$
 
 \item Let $t \in \bZ$. The functors $F = V_t \otimes(\cdot)$, $E =V_t^* \otimes(\cdot) $ induce an action of the Lie algebra $\sll_{\bZ}$ on the complexified Grothendieck group $\bC \otimes_{\bZ} Gr(\cV)$. 
 
 We then have an isomorphism of $\sll_{\bZ}$-modules $$\bC \otimes_{\bZ} Gr(\cV) \stackrel{\sim{}}{\longrightarrow} \mathfrak{F} \otimes \mathfrak{F}^{\vee}_t, \; \; \left[ \mathbf{V}(\lam) \right] \mapsto v_{\lam^{\bullet}} \otimes v_{\lam^{\circ}}$$
\end{enumerate}

\end{introtheorem}
In fact, we prove a stronger statement. Let $Pol$ be the category of polynomial representations of $\gl(\infty)$ (alternatively, this can be described as the category of all Schur functors, or the free Karoubian additive symmetric monoidal $\bC$-linear category with one generator). This category has an $\sll_{\bZ}$-action categorifying the Fock space representation $\mathfrak{F}$ (see \cite{HY, L}). One can similarly define a twisted $\sll_{\bZ}$-action on $Pol$ categorifying the Fock space $\mathfrak{F}^{\vee}_t$. By abuse of notation, we denote the category $Pol$ with these actions by $Pol$, $ Pol_t^{\vee}$ respectively. 

With this notation, we can now state our second result, showing that $\cV$ is a tensor product categorification $ $ in the sense of Losev and Webster (see \cite[Remark 3.6]{LW}, and details in Definition \ref{def:categorical_tensor_product}):

\begin{introtheorem}\label{introthrm:tens_prod_categorif}
Let $t \in \bZ$. The category $\cV$ with the action of $\sll_{\bZ}$ is a tensor 
product categorification $Pol \otimes Pol_t^{\vee}$.

\InnaC{Moreover, is unique in the following sense: consider a lower highest-weight category $\C$ (see Definition \ref{def:categorical_tensor_product}) with an $\sll_{\bZ}$-action $(F', E', \rx', \tau')$ making it a tensor product categorification $Pol \otimes Pol_t^{\vee}$ in the sense of Definition \ref{def:categorical_tensor_product}, and such that the natural transformation $\tau'_{E'F'}: E'F' \to F'E'$ induced by $\tau'$ is an isomorphism. 
 
 Then we have a strongly equivariant equivalence $\C \cong \cV$.}
\end{introtheorem}

\begin{remark}
 Of course, when $t \notin \bZ$, $\mathcal{D}_t$ is an (exterior) tensor product of the categorical $\sll_{\bZ}\times \InnaA{\sll_{\bZ}}$-module $Pol \boxtimes Pol^{\vee}_0$ in the sense of Losev and Webster. The uniqueness of such a tensor product categorification follows from the uniqueness of the categorification of the $\sll_{\bZ}$-representation $\mathfrak{F}$, which is straightforward (the weight spaces in $\mathfrak{F}$ being one-dimensional).
\end{remark}


In a separate result, we establish a connection between multiplicities in standard filtrations of tilting objects and multiplicities of the Jordan-Holder components of the standard modules.

\begin{introtheorem}\label{introthrm:equiv_multip}
 Let $t \in \bC$, and let $\lambda = (\lam^{\bullet}, \lam^{\circ}), \mu=(\mu^{\bullet}, \mu^{\circ}) $ be two bipartitions, regarded as pairs of Young diagrams. 
 
 Then $$\left( \mathbf{T}(\lam):\mathbf{V}(\mu) \right)_{\cV} = \left[ \mathbf{V}(\lam) :\mathbf{L}(\mu) \right]_{\cV}$$
 where the LHS denotes the multiplicity of the standard object $\mathbf{V}(\mu)$ in the standard filtration of the standard object $\mathbf{T}(\lam)$, and the RHS denotes the multiplicity of the simple object $\mathbf{L}(\lam)$ among the Jordan-Holder components of $\mathbf{V}(\mu)$.
 \end{introtheorem}
 \begin{remark}
  When $t \notin \bZ$, these multiplicities are $\delta_{\lam, \mu}$.
 \end{remark}

Moreover, we show that these multiplicities coincide with the multiplicity $D^{\lam}_{\mu}(t)$ appearing in \cite[Section 6]{CW}. This statement is related to the existence of an exact faithful SM functor $Rep^{alg}(GL_{\infty}) \to \cV$, $Rep^{alg}(GL_{\infty})$ being the category of algebraic representations of $GL_{\infty}$ as described in \cite{DPS, SS}. The standard objects $\mathbf{V}(\lam)$ in $\cV$ are images of the simple objects in the category $Rep^{alg}(GL_{\infty})$, and thus are ``flat'' with respect to the parameter $t$. This allows one to replace the ``lifting'' techniques from \cite{CO, CW} (passing from a specific value of the parameter $t$ to a generic value) by computations in the category $Rep^{alg}(GL_{\infty})$, aided by its embedding into the category $\mathcal{D}_t$ for generic $t$. 

This theorem also allows us to obtain partial results on the action of the functors $F_a, E_a$ on tilting objects $\mathbf{T}(\lam)$ (see Section \ref{ssec:action_F_a_tilting}).


\subsection{Acknowledgements}
The problem of studying categorical actions on Deligne categories was suggested 
by I. Losev. 
The author would like to thank I. Losev for his suggestion and for sharing his 
ideas on the subject, and P. Etingof and V. Serganova for their helpful 
comments and suggestions. 
The author is extremely grateful to J. Brundan for his thorough explanations on 
categorical actions and pointing out an error in the earlier version of the 
article. This project has received funding from the ERC under grant agreement No. 669655 (PI: David Kazhdan).
 
\section{Notation}\label{sec:notn}
The base field throughout the paper will be $\bC$. All the categories considered will be $\bC$-linear.
\subsection{Young diagrams and partitions}\label{ssec:notn_Young}
We will use the notions of partition and Young diagram interchangeably, and denote these by small-case greek letters. The set of all Young diagrams will be denoted by $\mathcal{P}$.

A {\it bipartition} is a pair of partitions, denoted by $\lam = (\lam^{\bullet}, \lam^{\circ}) \in \mathcal{P} \times \mathcal{P}$. 

Given a partition $\nu \vdash n$, we denote by $\nu_i, i\geq 0$ the parts of $\nu$ in non-increasing order ($\nu_i=0$ for $i>>0$). We also denote $\abs{\nu} :=n$, $\ell(\nu) := \max\{i: \nu_i \neq 0\}$, and by $\nu^{\vee}$ the transpose of $\nu$ (e.g. the transpose of $\yng(2)$ is $\yng(1,1)$). 

For a bipartition $\lam = (\lam^{\bullet}, \lam^{\circ})$, we denote $\abs{\lam} := \abs{\lam^{\bullet}} + \abs{\lam^{\circ}}$.

Given a cell $(i, j)$ in the Young diagram of a bipartition $\nu$, we denote the content of this cell by $ct(\square(i, j)):= i-j$.

Given a Young diagram $\nu$, we denote by $\nu + \square $ (resp. $\nu - \square$) the set of all Young diagrams obtained by adding (resp. removing) a box from $\nu$. The sets $\nu \pm \square $ are disjoint unions of subsets $\nu \pm \square_a$ ($a\in \bC$); the subset $\nu \pm \square_a$ contains the Young diagrams obtained by adding/removing a box of content $a$ from $\nu$ (it is empty if $a \notin \bZ$, otherwise it contains at most one element). 

When considering a bipartition $\lam = (\lam^{\bullet}, \lam^{\circ})$, we will 
use the notation $$\lam \pm \InnaD{\young(\bullet)}_a := \{\mu \in \mathcal{P} 
\times 
\mathcal{P}: \mu^{\bullet} \in \lam^{\bullet} \pm \InnaD{\young(\bullet)}_a, \; 
\mu^{\circ} 
= \lam^{\circ} \}$$ and $$\lam \pm \square_a := \{\mu \in \mathcal{P} \times 
\mathcal{P}: \mu^{\circ} \in \lam^{\circ} \InnaD{\pm} \square_a, \; 
\mu^{\bullet} = \lam^{\bullet}\}$$

Whenever a set of objects $\{ X_{\lambda} \}$ in an additive category is parametrized by partitions (or bi-partitions), we denote by $X_{\nu \pm \square_a}$ the direct sum of objects corresponding to the partitions (or bi-partitions) in the set $\nu \pm \square_a$. In particular, if the set $\nu \pm \square_a$ is empty, then the object is zero.

\subsection{Tensor categories}\label{ssec:notn_tensor_cat}
We will use the notions of a rigid symmetric monoidal $\bC$-linear category (symmetric monoidal will be abbreviated as SM) and of symmetric monoidal functors (SM functor for short) as defined in 
\cite{D1, EGNO}. In such a category $\mathcal{C}$ we denote by $\triv $ the unit object, by $\sigma_{C_1, C_2}$ the symmetry morphisms, and by $ev, coev$ the evaluation and coevaluation morphisms.

A (symmetric) {\it tensor category} (after \cite{D1, EGNO}) will stand for an abelian rigid SM category where the bifunctor $\otimes$ is bilinear on 
morphisms and where
$\End_{\C}(\triv) = \bC$.
\subsection{Representations of a the Lie algebra object \texorpdfstring{$\gl(V)$}{gl}}\label{ssec:notn_Lie_alg_obj}
Let $\C$ be a rigid SM category, and let $V \in \C$. The 
object $\gl(V) := V\otimes V^*$ is a Lie algebra object, and one can consider the 
category $\Repc(\gl(V))$ of representations of $\gl(V)$ in $\C$, which is again a rigid SM category.

The objects of $\Repc(\gl(V))$ are pairs $(M, \rho)$ where $\rho: \gl(V) \otimes 
M \longrightarrow M$. In particular, we have objects $V$, $V^*$, the standard and co-standard representations of $\gl(V)$. 


\subsection{Representations of the general linear Lie superalgebra}\label{ssec:notn_rep_superalg}
Throughout the paper, we will consider the tensor category $Rep(\gl(m|n))$ of 
finite-dimensional representations of the general linear Lie superalgebra 
$\gl(m|n)$ (together with even morphisms) in the sense of \cite[Section 
4]{EHS}. 
This category is ``half'' of the usual category of finite-dimensional 
representations of $\gl(m|n)$ which are integrable over $GL(m|n)$; it contains 
those representations of $\gl(m|n)$ on which the $\bZ_2$-grading is given by the 
action of the element $diag(1, \ldots, 1, -1, \ldots, -1)$ from the supergroup 
$GL(m|n)$. Any object in the category $Rep(\gl(m|n))$ is a subquotient of a 
finite direct sum of mixed tensor powers of $\bC^{m|n}$.

The category $Rep(\gl(m|n))$ is a highest weight category. 

To describe the weight poset of this category, let us introduce the following 
notation:
\begin{itemize}
 \item Let $w_1, \ldots, w_m, w_{m+1}, \ldots, w_{m+n}$ be the basis of the 
standard representation $W:=\bC^{m|n}$.
 \item Let $E_{i,j}$, $1 \leq i, j \leq m+n$ be the corresponding matrix units 
basis of $\gl(m|n)$.
 \item Let $\mathfrak{h}$ be the Cartan subalgebra consisting of diagonal 
matrices, and consider the Borel subalgebra of upper-triangular matrices in the 
above basis.
 \item Let $\delta_1, \ldots, \delta_{m+n}$ be the basis of $\mathfrak{h}^*$ 
dual to $E_{1,1}, \ldots, E_{m+n, m+n}$. The root system is $\{\delta_i - 
\delta_j \rvert i \neq j\}$, with positive roots $\delta_i - \delta_j$, $i <j$. 
\end{itemize}

The simple modules $\mathit{L}(a_1, a_2, \ldots,a_m, b_1, b_2, \ldots, b_n)$ are 
then parametrized by pairs of weakly-decreasing integer sequences $a_1, a_2, 
\ldots,a_m$, $ b_1, b_2, \ldots, b_n$. The highest weight of such a module is 
$\sum_{i=1}^m a_i\delta_i+\sum_{j=1}^n b_{j}\delta_{m+j}$.

We denote the standard objects in the category $Rep(\gl(m|n))$ by 
$\mathit{K}(a_1, a_2, \ldots,a_m, b_1, b_2, \ldots, b_n)$ (these are called {\it 
Kac modules}).

We now describe two (partial) correspondences between bipartitions and simple 
$\gl(m|n)$-modules. For each bipartition $\lam:=(\lam^{\bullet}, \lam^{\circ})$ 
we define a $\gl(m|n)$-module $\mathit{L}(\lam)$, which is either simple or 
zero.

{\bf Case $m, n \neq 0$.} Given a bipartition $\lambda$ such that 
$\ell(\lam^{\bullet}) \leq m, \ell(\lam^{\circ})\leq n$, let $$\mathit{L}(\lam): 
= \mathit{L}(a_1, a_2, \ldots,a_m, b_1, b_2, \ldots, b_n)$$ and
$$\mathit{K}(\lam): = \mathit{K}(a_1, a_2, \ldots,a_m, b_1, b_2, \ldots, b_n)$$
where $a_i := \lam^{\bullet}_i$, $b_i:= - \lam^{\circ \vee}_{n-i+1}$ for any 
$i$. 

We will denote: $\vec{\lam}:=\sum_{i=1}^m a_i\delta_i+\sum_{j=1}^n 
b_{j}\delta_{m+j}$.

For other bipartitions we put $\mathit{L}(\lam):=0$, $\mathit{K}(\lam):=0$.

{\bf Case $n=0$.} In this case $Rep(\gl(m|0))$ is the semisimple category of all 
rational representations of $GL(m)$.
Given a bipartition $\lambda$ such that $\ell(\lam^{\bullet}) + 
\ell(\lam^{\circ}) \leq m$, we define:
$$\mathit{L}(\lam): = \mathit{L}(a_1, a_2, \ldots,a_m)$$ where $a_i := 
\lam^{\bullet}_i$ for $i\leq \ell(\lam^{\bullet})$, and $a_{m-i+1} := - 
\lam^{\circ}_i$ for $i \leq \ell(\lam^{\circ})$. 

We will denote: $\vec{\lam}:=\sum_{i=1}^m a_i\delta_i$.

For other bipartitions we put $\mathit{L}(\lam):=0$.

\begin{example}
 The bipartition $\lam = (\young(\bullet,\bullet), \yng(1,1))$ 
corresponds to 
$\mathit{L}(1,1,0,\ldots,0,-1, -1)$ if $m\geq 4$, $n=0$. On the other hand, 
$\lam$ corresponds to $\mathit{L}(1,1,0,\ldots,0,-2)$ if $m, n\geq 2$.
\end{example}

This notation is convenient in view of the (straightforward) lemma below:
\begin{lemma}
 Let $\lambda:=(\lam^{\bullet}, \lam^{\circ})$ be a bipartition, and consider 
$m, n$ such that either $\ell(\lam^{\bullet}) \leq m, \ell(\lam^{\circ})\leq n$ 
or $\ell(\lam^{\bullet}) + \ell(\lam^{\circ}) \leq m$, $n=0$.
 
 For any such $m, n$, let $W = \bC^{m|n}$ be the standard representation of 
$\gl(m|n)$. Then the maximal weight of $S^{\lam^{\bullet}} W \otimes 
S^{\lam^{\circ}} W^*$ is $$\vec{\lam}:=\sum_{i=1}^m a_i\delta_i+\sum_{j=1}^n 
b_{j}\delta_{m+j}$$ where $a_i, b_i$ are defined above. This weight occurs with 
multiplicity $1$. 
 
 In other words, there exists a unique (up to scalar) non-zero homomorphism 
$$\mathit{K}(a_1, \ldots, a_m, b_1, \ldots, b_n) \longrightarrow 
S^{\lam^{\bullet}} W \otimes S^{\lam^{\circ}} W^*$$
\end{lemma}
%
%
\section{Preliminaries}\label{sec:prelim} 

\subsection{Deligne categories 
\texorpdfstring{$\mathcal{D}_t$}{}}\label{ssec:Deligne_cat}

We briefly recall some facts about the Deligne categories $\mathcal{D}_t$ (also 
known as $\uRep(GL_t)$). 

Let $t \in \bC$.
\subsubsection{Construction}\label{sssec:Deligne_constr}
Consider the free SM $\bC$-linear category $\mathcal{OB}$ generated by one 
dualizable object $V$; this is the oriented Brauer category, as described in 
\cite{BCNR}. Next, consider the quotient ${\cD^{0}_t}$ of this category by the 
relation $\dim(V) =t$ (for any $t \in \bC$). The category ${\cD^{0}_t}$ (denoted 
by Deligne as $\uRep_0(\gl_t)$, see \cite[Section 10]{D2}) is freely generated, 
as a SM $\bC$-linear category, by one dualizable object of dimension $t$. We 
denote by $V_t$ and $V^*_t$ the $t$-dimensional generator and its dual. The 
objects in this category will be mixed tensor powers $V_t^{\otimes p} \otimes 
{V_t^*}^{\otimes q}$, and the $\Hom$-spaces will have a diagrammatic basis, 
generated by tensor products of the morphisms $\id_{V_t}, \sigma_{V_t, V_t}, 
\sigma_{V_t, V_t^*}, ev_{V_t}, coev_{V_t}$ and their duals, with the relation 
$$\dim V_t :=  ev_{V_t} \circ \sigma_{V_t, V_t^*} \circ coev_{V_t} = t.$$ 

We will denote by $\mathcal{D}_t$ the Karoubi additive envelope of the above 
``free'' category, which is obtained  by adding formal direct sums and images of 
idempotents. The category $\mathcal{D}_t$ is a Karoubian additive rigid SM 
category, also called the Deligne's category $\uRep(GL_t)$. Its structure is 
studied in \cite{D2, CW}, and it is that the indecomposable objects of 
$\mathcal{D}_t$ (up to isomorphism) are parametrized by the set of all 
bipartitions. We will denote the indecomposable object corresponding to the 
bipartition $\lam$ by $\mathbf{T}(\lam)$.

\begin{example}
 For any $t \in \bC$, $\mathbf{T}(\InnaC{\varnothing}, \InnaC{\varnothing}) = 
\triv$, $\mathbf{T}(\InnaD{\young(\bullet)}, \InnaC{\varnothing}) = V_t$, 
$\mathbf{T}( 
\InnaC{\varnothing}, \yng(1)) = V_t^*$. 
 
 When $t \neq 0$, we have: $\mathbf{T}(\InnaC{\varnothing}, \InnaC{\varnothing}) 
\oplus \mathbf{T}(\InnaD{\young(\bullet)}, \yng(1)) = V_t \otimes V_t^*$. On the 
other hand, 
when $t=0$, we have: $\mathbf{T}(\InnaD{\young(\bullet)}, \yng(1)) = V_t \otimes 
V_t^*$.
\end{example}

\subsubsection{Universal property}\label{sssec:univer_prop_Deligne}
The category $\mathcal{D}_t$ possesses a universality property:
given a Karoubian additive rigid SM category $\C$, together with a fixed object 
$V \in \C$ of dimension $t$, there exists an essentially unique additive SM 
functor $\mathbf{S}_V: \mathcal{D}_t \longrightarrow \C$ such that $V_t \mapsto 
V$.

In particular, for $t \in \bZ$ and for any $m, n \in \bZ_+$ such that $m-n =t$, 
there exists a SM functor $$\mathbf{S}_{m, n}: \mathcal{D}_t \longrightarrow 
Rep(\gl(m|n)), \;\; V_t \mapsto \bC^{m|n}$$

In the special case $t = d \in \bZ_+$, the functor $\mathbf{S}_{d, 0}$ takes the 
indecomposable object $\mathbf{T}(\lam)$ to the $GL(d)$-module 
$\mathit{L}(\lam)$.

\subsubsection{Abelian version}\label{sssec:abelian_Deligne}
The category $\mathcal{D}_t$ is a priori not necessarily abelian (thus not a 
tensor category). Yet it turns out that for $t \notin \bZ$, this category is 
indeed abelian and even semisimple. 

For $t \in \bZ$, this is not the case. Fortunately, it turns out the one can 
construct a tensor category $\cV$ into which $\mathcal{D}_t$ embeds as a full 
rigid SM subcategory. We will set $\cV := \mathcal{D}_t$ when $t \notin \bZ$. 

The collection $\{\cV, Rep(\gl(m|n)) \rvert m-n =t\}$ acts as a system of 
abelian envelopes of the Deligne category $\mathcal{D}_t$ (see \cite{EHS} for 
details).

 When $t\in \bZ$, the category $\cV$ is constructed in \cite{EHS} as a certain 
limit of categories $Rep(\gl(m|n))$ with $m-n = t$. For any $m, n$ such that 
$m-n =t$, there exists an additive SM functor (not exact!) $$\mathcal{F}_{m|n}: 
\cV \longrightarrow Rep(\gl(m|n))$$ such that the composition of $ 
\mathcal{F}_{m|n}$ with the embedding $\mathcal{D}_t \hookrightarrow \cV$. 

The functors $\mathcal{F}_{m|n}$ are ``local'' equivalences: the categories 
$\cV$, $Rep(\gl(m|n))$ have natural $\bZ_+$-filtrations on objects which are 
preserved by the functors $\mathcal{F}_{m|n}$, and the latter induce an 
equivalence between each filtra $\mathcal{V}_t^k$ and $Rep^k(\gl(m|n))$ for $m, 
n>>k$. 

In the category $\cV$ we distinguish three types of objects, all defined up to 
isomorphism: the simple objects, denoted by $\mathbf{L}(\lam)$, the ``standard 
objects'', denoted by $\mathbf{V}(\lam)$ and the objects $\mathbf{T}(\lam)$. The 
three types are parametrized by arbitrary bipartitions. The name ``standard'' 
comes from the fact that the subcategories $\mathcal{V}_t^k$ are highest-weight 
categories, with objects $\mathbf{L}(\lam), \mathbf{V}(\lam), \mathbf{T}(\lam)$ 
(for $\abs{\lam^{\bullet}}+ \abs{\lam^{\circ}} \leq k$) playing the roles of 
simples, standard and tilting objects respectively\footnote{See \cite[Section 
8.5]{EHS} about the ``local'' highest weight structure of $\cV$}. 

\begin{remark}
 When $t \notin \bZ$, the objects $\mathbf{L}(\lam), \mathbf{V}(\lam), 
\mathbf{T}(\lam)$ are defined to be isomorphic.
\end{remark}

Finally, let us say a few words about the $\gl(m|n)$-modules 
$\mathit{V}(\lam):=\mathcal{F}_{m|n}(\mathbf{V}(\lam))$ in $Rep^k(\gl(m|n))$ 
(here $k \geq \abs{\lam}$ is fixed, $m, n >>k$, $m-n=t$). These modules appear 
as submodules of $S^{\lam^{\bullet}} W \otimes S^{\lam^{\circ}} W^*$, where $W = 
\bC^{m|n}$ is the standard representation of $\gl(m|n)$. 

In fact, we have the following description (see \cite{EHS}):
\begin{lemma}\label{lem:V_lam_highest_weight}
 $\mathit{V}(\lam)$ is a highest weight module, with highest weight 
$\vec{\lam}:=\sum_{i=1}^m a_i\delta_i+\sum_{j=1}^n b_{j}\delta_{m+j}$ where $a_i 
:= \lam^{\bullet}_i$, $b_i:= -\lam^{\circ \vee}_{n-i+1}$ for any $i$.

 In other words, there exists a unique (up to scalar) non-zero homomorphism 
$$\InnaD{\mathit{K}(a_1, \ldots, a_m, b_1, \ldots, b_n)} \longrightarrow 
S^{\lam^{\bullet}} W \otimes S^{\lam^{\circ}} 
W^*$$ with image $\mathit{V}(\lam)$. In fact, $\mathit{V}(\lam)$ is the maximal 
quotient of $\mathit{K}(a_1, \ldots, a_m, b_1, \ldots, b_n)$ lying in 
$Rep^k(\gl(m|n))$.
\end{lemma}

\subsubsection{Deligne categories with formal 
parameter}\label{sssec:formal_param}

One can construct a similar Deligne category for a formal variable $T$ instead 
of complex parameter $t$; the obtained category $\mathcal{D}_T$ is a 
$\bC\InnaD{((T-t))}$-linear semisimple tensor category. The simple objects (up 
to 
isomorphism) are again labeled by all bipartitions; by abuse of notation, we 
will denote them by $\mathbf{T}(\mu)$ as well.

The connection between the Deligne categories $\mathcal{D}_t$ and 
$\mathcal{D}_T$ is manifested by the existence of a homomorphism of rings with 
scalar products
$$Lift_t: K_0(\mathcal{D}_t) \rightarrow K_0(\mathcal{D}_T)$$ ($K_0$ stands for 
the split Grothendieck ring) which respects duals and takes $V_t$ to $V_T$. This 
map was defined in \cite[Section 6]{CW}.

The existence of this map is based on the fact that for any $r', s' \in \bZ_+$, 
any idempotent $e$ in the walled Brauer algebra $$Br_{\bC}(r', s') = 
\End_{\mathcal{D}_t}(V_t^{\otimes r'} \otimes V_t^{* \otimes s'})$$ over $\bC$ 
can be lifted to an idempotent $\tilde{e}$ in the walled Brauer algebra 
$Br_{\bC\InnaD{((T-t))}}(r', s') = \End_{\mathcal{D}_T}(V_T^{\otimes r'} 
\otimes 
V_T^{\otimes s'})$ over $\bC\InnaD{((T-t))}$, so that $\tilde{e}\rvert_{T=t} = 
e$ (cf. 
\cite[Section 3.2]{CO}). 
idempotents 
We will denote by $D^{\lam}_{\mu}(t)$ the multiplicity of $\mathbf{T}(\mu)$ 
inside $Lift_t(\mathbf{T}(\lam))$. This multiplicity has been shown to be either 
$1$ or $0$, and \cite[Section 6]{CW} contains an explicit algorithm for 
computing $D^{\lam}_{\mu}$ (see also Section \ref{sec:mult}).

For a fixed bipartition $\mu$ and $t \notin \{0, \pm 1, \pm 2, \ldots, \pm 
\abs{\mu^{\bullet}} + \abs{\mu^{\circ}} \}$, it was proved that 
$Lift_t(\mathbf{T}(\mu)) =\mathbf{T}(\mu)$, and thus $D^{\lam}_{\mu}(t) = 
\delta_{\lam, \mu}$ for almost all $t$.

\subsection{Lie algebra \texorpdfstring{$\sll_{\bZ}$}{sl 
infinity}}\label{ssec:sl_infty}
Let $\sll_{\bZ}$ be the Lie algebra of $\bZ \times \bZ$-matrices with 
finitely-many non-zero entries and trace zero. This Lie algebra is generated by 
the elements $f_a := E_{a+1, a}$ and $e_a:=E_{a, a+1}$, where $a \in \bZ$.

We will use the same definitions as in \cite[Section 2.2]{BLW}. The weight 
lattice of $\sll_{\bZ}$ is $\bigoplus_{a \in \bZ} \bZ \varpi_a$ where $\varpi_a$ 
is the $a$-th fundamental weight. The $a$-th simple root is $$\alpha_a = 
2\varpi_a - \varpi_{a-1} - \varpi_{a+1}$$ These span the root lattice inside the 
weight lattice. We define a dominance order $\geq$ on the lattice of weights by 
setting $\beta \geq \gamma$ whenever $\beta - \gamma$ is a finite combination of 
simple roots with non-negative integral coefficients.

We denote by $\bC^{\bZ} := span_{\bC} \{u_i \rvert i \in \bZ\}$ the tautological 
representation of $\sll_{\bZ}$. We have: $$f_a.u_i = \delta_{a, i} u_{a+1}, 
\;\;\; e_a.u_i = \delta_{a+1, i} u_{a}$$

The weight of the vector $u_i$ is $\varpi_i - \varpi_{i-1}$.

\subsection{Fock space}\label{ssec:Fock_space}

The {\it Fock space} is a $\bC$-vector space with a basis consisting of infinite 
wedges $u_I:= u_{i_0} \wedge u_{i_{-1}}   \wedge u_{i_{-2}}  \wedge \ldots$ 
($i_{-s} \in \bZ$ for any $s$) where $I:=(i_0, i_{-1}, i_{-2},  \ldots)$ is a 
strictly decreasing infinite sequence satisfying: $i_{-s} = -s$ for $s >>0$.

This space has an obvious action of $\sll_{\bZ}$ on it. 

We will also use another notation for this basis, indexing by partitions of 
arbitrary size:
$$v_{\nu} := u_{\nu_1} \wedge u_{\nu_2 -1} \wedge \ldots \wedge u_{\nu_s -s +1} 
\wedge\ldots$$
In the new notation the action of $\sll_{\bZ}$ on $\mathfrak{F}$ is given by 
$$f_a.v_{\nu} := v_{\nu + \square_a}, \; \; e_a.v_{\nu} := v_{\nu - \square_a}$$ 
for any partition $\nu $ and any $a \in \bZ$.

Note that $v_{\nu} \in \mathfrak{F}$ is a weight vector, i.e. $h_a := [e_a, f_a] 
\in \sll_{\bZ}$ acts on $v_{\nu} \in \mathfrak{F}$ by a scalar $n_a(\nu) \in \{ 
0, \pm 1\}$, where $$n_a(\nu) = \begin{cases}
1  &\text{  if  } \nu + \square_a \neq \emptyset, \\
-1 &\text{  if  } \nu - \square_a \neq \emptyset, 
\\ 0 &\text{  else}

                             \end{cases}
$$ We denote the weight $(n_a(\nu))_{a\in \bZ}$ of $v_{\nu}$ by $\omega_{\nu}$.

\subsubsection{Twisted and shifted duals}\label{sssec:Fock_duals}
We will also consider the {\it twisted dual Fock space representation 
$\mathfrak{F}^{\vee}$} of $\sll_{\bZ}$.

The space $\mathfrak{F}^{\vee}$ is isomorphic to $\mathfrak{F}$ as a vector 
space, with the action of $\sll_{\bZ}$ given by $f_a.v_{\nu} := v_{\nu- 
\square_{-a}}$, $ e_a.v_{\nu} :=v_{\nu+ \square_{-a}}$.

Finally, \InnaA{for $t \in \bZ$}, we will consider the \InnaA{``shifted dual'' 
representation $\mathfrak{F}^{\vee}_t$ of $ \sll_{\bZ}$: the underlying vector 
space is $\mathfrak{F}^{\vee}$, with $f_a, e_a$ ($a \in \bZ$) acting by 
$f_a.v_{\nu} := v_{\nu- \square_{-\InnaC{(a+t)}}}$, $ e_a.v_{\nu} := v_{\nu+ 
\square_{-\InnaC{(a+t)}}}$.}

\subsubsection{Fock space as a ``stable'' inverse 
limit}\label{sssec:Fock_inv_lim}
We will now show that the space $\mathfrak{F}$ is a ``stable'' inverse limit of 
a system of subspaces of the exterior powers $\wedge^n \bC^{\bZ}$, $n \geq 0$. 
We will explain the precise meaning below.

The space $\wedge^n \bC^{\bZ}$ has a basis $u_I:= u_{i_0} \wedge u_{i_{-1}}   
\wedge \ldots  \wedge u_{i_{-n}}$ indexed by descreasing sequences $I$ of length 
$n$. We have $\sll_{\bZ}$-equivariant maps
$$ \phi_{n+1}:\wedge^{n+1} \bC^{\bZ} \longrightarrow \wedge^n \bC^{\bZ}, \;\; 
u_{i_0} \wedge u_{i_{-1}} \wedge \ldots \wedge u_{i_{-n-1}} \mapsto u_{i_0} 
\wedge u_{i_{-1}}   \wedge \ldots  \wedge u_{i_{-n}}$$ and
$$\pi_n:\mathfrak{F} \longrightarrow \wedge^n \bC^{\bZ}, \;\; u_I \mapsto 
u_{i_0} \wedge u_{i_{-1}}   \wedge \ldots  \wedge u_{i_{-n}}$$

Now, these maps define an inverse limit $\varprojlim_{n \to \infty} \wedge^n 
\bC^{\bZ}$ and a map $ \mathfrak{F} \longrightarrow \varprojlim_{n \to \infty} 
\wedge^n \bC^{\bZ}$, but this map will not be an isomorphism (e.g. $u_1 \wedge 
u_0 \wedge u_{-1} \wedge \ldots$ will not lie in the image).

To fix this, consider a system of $\bZ_+$-parametrized subspaces in $\wedge^n 
\bC^{\bZ}$ and in $\mathfrak{F}$ given by the so-called ``energy'' function on 
the basis:

Let the {\it energy} of the vector $u_I$ be $\sum_s i_{-s} +s$. Then the 
subspaces $\wedge^n \bC^{\bZ}_{(k)}$ and $\mathfrak{F}_{(k)}$ are defined as the 
span of vectors $u_I$ of energy at most $k$, for which $i_{-s} +s \geq 0$ for 
all $s$.
\begin{remark}
  In the latter case we have: $\mathfrak{F}_{(k)} = span_{\bC}\{v_{\lam} \rvert 
\abs{\lam} \leq k \}$.

                                                       \end{remark}
\begin{remark}
 Notice that the above subspaces are not preserved by the action of 
$\sll_{\bZ}$.
\end{remark}

  Then $\mathfrak{F} = \varinjlim_{k \to \infty} \mathfrak{F}_{(k)}$, and the 
maps $\phi_{n+1}$, $\pi_{n}$ preserve these systems of subspaces. Moreover, 
$\pi_n \rvert_{\mathfrak{F}_{(k)}}$, $ \phi_{n+1}\rvert_{ \wedge^{n+1} 
\bC^{\bZ}_{(k)}}$ become isomorphisms for $n \geq k-1$, and thus 
\begin{equation}\label{eq:Fock_limit}
   \mathfrak{F} \cong \varinjlim_{k \to \infty} \varprojlim_{n \to \infty} 
\wedge^n \bC^{\bZ}_{(k)}  \end{equation}
Note that this is not an isomorphism of $\sll_{\infty}$-modules.
\begin{remark}
 This isomorphism has a categorical version, describing the category $Pol$ of 
``stable'' polynomial $GL$-modules as a special limit of the categories of 
polynomial $GL_n$-modules. See \cite{HY} for details.
\end{remark}


\section{Weight diagrams}\label{sec:weight_diag}
In this section we briefly recall the definitions of weight diagrams given in 
\cite{EHS}, \cite{CW}. The weight diagrams were originally defined and used in 
the representation theory of superalgebras and Khovanov arc algebras (see 
\cite{BS, MS}), and provide a powerful combinatorial tool when studying the 
highest-weight structure of the categories $Rep(\gl(m|n))$ and (locally) $\cV$.

\subsection{Two types of diagrams}
Let $\lambda$ be a bipartition, and let $t \in \bC$. Below we define two 
diagrams corresponding to $\lam$, called $d_{\lam}$ and $d'_{\lam}$, which 
represent a labeling of integers by symbols $\times, \bigcirc, <, >$. 

For the weight diagram $d_{\lam}$, we will use the same notation as in 
\cite{EHS}. This notation differs slightly from the notation in \cite{BS, MS}: 
in \cite{BS} the symbols are permuted (see \cite{MS} for a dictionary) and the 
diagrams are shifted.

Consider two infinite sequences $C:=\{\lambda^{\bullet}_i+t-i \}_{i \geq 1}$, 
$D:=\{\lambda^{\circ}_i-i\}_{i \geq 1}$. Consider the map $f_\lambda:\mathbb 
Z\to\{\times,>,<,\circ\}$ defined by
\begin{equation}\label{eq:weightdiagrams}
f_\lambda(s)=\begin{cases} \bigcirc,\,\text{if}\,s \notin C \cup D \\
>,\,\text{if}\,s \in C \setminus D\\
<,\,\text{if}\,s \in D \setminus C\\
\times,\,\text{if}\,s \in C \cap D.
\end{cases}
\end{equation} To the map $f_{\lam}$ we associate a diagram $d_{\lam}$ where 
every integer $s$ is labeled by $f_{\lam}(s)$.
\begin{example}
 Let $t=0$, $\lam = (\InnaC{\varnothing}, \InnaC{\varnothing})$. The diagram 
$d_{\lam}$ is then

 $$ 
 \xymatrix{ &\stackrel{\times}{-5} &\stackrel{\times}{-4} &\stackrel{\times}{-3} 
&\stackrel{\times}{-2}  &\stackrel{\times}{-1} &\stackrel{\bigcirc}{0}  
&\stackrel{\bigcirc}{1} &\stackrel{\bigcirc}{2} &\stackrel{\bigcirc}{3} 
&\stackrel{\bigcirc}{4} &\stackrel{\bigcirc}{5} }$$

\end{example}
\begin{example}
 Let $t=1$, $\lam = (\InnaD{\young(\bullet\bullet)}, \yng(2))$. The diagram 
$d_{\lam}$ is then

 $$ 
 \xymatrix{ &\stackrel{\times}{-5} &\stackrel{\times}{-4} &\stackrel{\times}{-3} 
&\stackrel{\times}{-2}  &\stackrel{>}{-1} &\stackrel{\bigcirc}{0}  
&\stackrel{<}{1} &\stackrel{>}{2} &\stackrel{\bigcirc}{3} 
&\stackrel{\bigcirc}{4} &\stackrel{\bigcirc}{5} }$$

\end{example}
Note that $f_\lambda(s)=\bigcirc$ for all $s>>0$. If $t \notin \bZ$, then the 
diagram $d_{\lam}$ contains only the symbols $\bigcirc, <$. If $t \in \bZ$, then 
$f_\lambda(s)=\times$ for all $s<<0$.

\mbox{}

Next, consider the (infinite) sets $C' := \{\lambda^{\bullet}_i+t-i \}_{i \geq 
1}$, $D':=\bZ \setminus \{i-\lambda^{\circ}_i-1\}_{i \geq 1}$. 

Consider the map $f'_\lambda:\mathbb Z\to\{\times,>,<,\circ\}$ defined by
\begin{equation}\label{eq:CWdiagrams}
f'_\lambda(s)=\begin{cases} \bigcirc,\,\text{if}\,s \notin C' \cup D' \\
>,\,\text{if}\,s \in C' \setminus D'\\
<,\,\text{if}\,s \in D' \setminus C'\\
\times,\,\text{if}\,s \in C \cap D.
\end{cases}
\end{equation} 
To the map $f'_{\lam}$ we associate a diagram $d'_{\lam}$ where every integer 
$s$ is labeled by $f'_{\lam}(s)$.
\begin{example}
 Let $t=0$, $\lam = (\InnaC{\varnothing}, \InnaC{\varnothing})$. The diagram 
$d'_{\lam}$ is then

 $$ 
 \xymatrix{ &\stackrel{\times}{-5} &\stackrel{\times}{-4} &\stackrel{\times}{-3} 
&\stackrel{\times}{-2}  &\stackrel{\times}{-1} &\stackrel{\bigcirc}{0}  
&\stackrel{\bigcirc}{1} &\stackrel{\bigcirc}{2} &\stackrel{\bigcirc}{3} 
&\stackrel{\bigcirc}{4} &\stackrel{\bigcirc}{5} }$$

\end{example}
\begin{example}
 Let $t=1$, $\lam = (\InnaD{\young(\bullet\bullet)}, \yng(2))$. The diagram 
$d'_{\lam}$ is then

 $$ 
 \xymatrix{ &\stackrel{\times}{-5} &\stackrel{\times}{-4}  
&\stackrel{\times}{-3}  &\stackrel{>}{-2} &\stackrel{\times}{-1}  
&\stackrel{<}{0} &\stackrel{\bigcirc}{1} &\stackrel{>}{2} 
&\stackrel{\bigcirc}{3} &\stackrel{\bigcirc}{4} &\stackrel{\bigcirc}{5}  }$$

\end{example}
\begin{example}
 Let $t=1$, $\lam = (\InnaD{\young(\bullet\bullet)}, \yng(1,1))$. The diagram 
$d'_{\lam}$ is then 

 $$ 
 \xymatrix{ &\stackrel{\times}{-5} &\stackrel{\times}{-4} &\stackrel{\times}{-3} 
&\stackrel{\times}{-2}  &\stackrel{>}{-1} &\stackrel{\bigcirc}{0}  
&\stackrel{<}{1} &\stackrel{>}{2} &\stackrel{\bigcirc}{3} 
&\stackrel{\bigcirc}{4} &\stackrel{\bigcirc}{5} }$$
\end{example}
\begin{remark}
 The definition of $d'_{\lam}$ corresponds to the diagrams (shifted by $t$) 
defined in \cite[Section 6]{CW}, but with the symbols permuted. Here is a 
dictionary: 
 
 \begin{center}
   \begin{tabular}{l||c c c c}
    Notation of \cite{CW} &$\vee$ &$\wedge$ &$\times$ &$\bigcirc$  \\ \hline

    Our notation &$\bigcirc$ &$\times$ &$>$ &$<$                                 
                                                                                 
                     \end{tabular}
 \end{center}

\end{remark}
\subsection{Relation between the two weight diagrams}
\begin{notation}\label{notn:transp_bipart}
 Let $\lambda, \mu$ be a bipartition. We denote by 
$\lam^{\vee}:=(\lam^{\bullet}, \lam^{\circ \vee})$, where $\lam^{\circ \vee}$ is 
the transposed Young diagram $\lam^{\circ}$.
\end{notation}
\begin{lemma}\label{lem:CW_vs_V}
 The diagrams $d_{\lam}$ and $d'_{\lam^{\vee}}$ coincide.
\end{lemma}
\begin{proof}
 We need to show that the set $D$ for $\lam$ and the set $D'$ for $\lam^{\vee}$ 
coincide. That is, that we want to show that for any Young diagram $\kappa:= 
\lambda^\circ$, we have $$\{ \kappa_i - i\}_{i \geq 1}\sqcup 
\{i-\kappa^\vee_i-1\}_{i \geq 1} = \bZ$$
 First, we check that these two sets do not intersect. Indeed, assume that 
$\kappa_i - i = j-\kappa^\vee_j-1$ for some $i, j$. Then we would have $\kappa_i 
+ \kappa^{\vee}_j+1 = i+j$. Yet this is impossible: $\kappa_i <j$ iff 
$\kappa^{\vee}_j <i$, and thus $\kappa_i + \kappa^{\vee}_j+1$ is either at most 
$i+j-1$, or at least $i+j+1$.
 
 Next, we show that the union of the two sets above is indeed $\bZ$. Let $N = 
\abs{\kappa}$. Then for any $j > N$, we have:
 $ j = (j+1) - \kappa_{j+1} - 1$, $ -j = \kappa_j - j$ (the first statement is 
true even for $j=N$). Thus $$\bZ \setminus \{-N, -N+1, \ldots, N-2, N-1\}  = \{j 
\rvert j\geq N\} \sqcup \{-j \rvert j>N\} \subset \{ \kappa_i - i\}_{i \geq 
1}\sqcup \{i-\kappa^\vee_i-1\}_{i \geq 1}$$
 
The set difference between the RHS and the LHS in the above inclusion is $\{ 
\kappa_i - i\}_{1 \leq i \leq N}\sqcup \{i-\kappa^\vee_i-1\}_{1 \leq i \leq N}$. 
This is a set of $2N$ elements, and thus it covers precisely the set $\{-N, 
-N+1, \ldots, N-2, N-1\}$.
 
\end{proof}

\subsection{Diagrams and the Fock space}
Finally, we prove a combinatorial auxiliary result relating weights in the Fock 
space representation of $\sll_{\bZ}$ and weight diagrams. This will be used to 
prove Theorem \ref{thrm:tens_prod_categorif}.

\begin{definition}\label{def:core_diag}
 The core $core(d'_{\lam})$  of the diagram $d'_{\lam}$ is the diagram obtained 
from $d'_{\lam}$ by replacing all the symbols $\times$ with $\circ$. The core 
$core(d_{\lam})$ of $d_{\lam}$ is defined in the same way.
\end{definition}

\begin{lemma}\label{lem:equal_cores_weights}
 Let $t \in \bC$. Let $\lambda, \mu$ be two bipartitions, such that 
$core(d'_{\lam}) = core(d'_{\mu})$. Then $$n_a(\lambda^{\bullet}) - 
n_{-\InnaC{(a+t)}}(\lambda^{\circ}) = n_a(\mu^{\bullet}) - 
n_{-\InnaC{(a+t)}}(\mu^{\circ})$$ for each $a \in \bZ$ where $n_a(\nu)$ equals 
$1$ if $\nu + \square_a \neq \emptyset$, $-1$ if $\nu - \square_a \neq 
\emptyset$, and zero otherwise, as defined in Section \ref{ssec:Fock_space}. 
\end{lemma}

\begin{proof}
First of all, recall that by Lemma \ref{lem:CW_vs_V}, $d'_{\lam} = 
d_{\lam^{\vee}}$, so $core(d_{\lam^{\vee}}) = core(d_{\mu^{\vee}})$, and 
$$n_a(\lambda^{\bullet}) - n_{-\InnaC{(a+t)}}(\lambda^{\circ}) = 
n_a(\lambda^{\bullet}) - n_{a+t}(\lambda^{\circ \vee}) $$
$$n_a(\mu^{\bullet}) - n_{-\InnaC{(a+t)}}(\mu^{\circ})= n_a(\mu^{\bullet}) - 
n_{a+t}(\mu^{\circ \vee})$$

Thus we wish to show that $core(d_{\lam^{\vee}}) = core(d_{\mu^{\vee}})$ implies 
$$n_a(\lambda^{\bullet}) - n_{a+t}(\lambda^{\circ \vee}) \stackrel{?}{=} 
n_a(\mu^{\bullet}) - n_{a+t}(\mu^{\circ \vee})$$
 Consider the sets $$C^{\lambda} := \{\lambda^{\bullet}_i - i\}, \;\; 
D^{\lambda}:= \{ \lambda^{\circ \vee}_i - i-t\}, \;\;\;\;\;\;\;\; C^{\mu} := 
\{\mu^{\bullet}_i - i\}, \; \; D^{\mu}:= \{ \mu^{\circ \vee}_i - i-t\}$$ (notice 
the shift by $t$ with respect to the previous definitions).
 
 Let $\phi^C_{\lambda}, \phi^D_{\lambda}: \bZ \to \{0, 1\}$ be the 
characteristic functions of the sets $C^{\lambda}$ and $D^{\lambda}$ 
respectively, and similarly for $\phi^C_{\mu}, \phi^D_{\mu}$.
 
 The condition $core(d_{\lambda^{\vee}}) = core(d_{\mu^{\vee}})$ implies that 
$$\phi^C_{\lambda}-  \phi^D_{\lambda} = \phi^C_{\mu} - \phi^D_{\mu}$$
 
 Now, $$n_a(\lambda^{\bullet}) = \begin{cases}
                                  1 &\text{  if  } a-1 \in C^{\lambda},  \;\; 
a\notin C^{\lambda} \\
                                  -1 &\text{  if  } a \in C^{\lambda},  \;\; a-1 
\notin C^{\lambda}\\
                                  0 &\text{  else  } 
                                 \end{cases}$$
so $$n_a(\lambda^{\bullet}) = \phi^C_{\lambda}(a-1) - \phi^C_{\lambda}(a)$$
Similarly, $$n_{a+t}(\lambda^{\circ \vee}) = \phi^D_{\lambda}(a-1) - 
\phi^D_{\lambda}(a)$$

Writing out similar identities for $\mu$, we obtain:
\begin{align*}
 &n_a(\lambda^{\bullet}) - n_{a+t}(\lambda^{\circ \vee}) = \phi^C_{\lambda}(a-1) 
- \phi^C_{\lambda}(a) - \left( \phi^D_{\lambda}(a-1) - \phi^D_{\lambda}(a) 
\right) = \\
 &= \phi^C_{\mu}(a-1) -  \phi^C_{\mu}(a) - \left( \phi^D_{\mu}(a-1) - 
\phi^D_{\mu}(a)\right) = n_a(\mu^{\bullet}) - n_{a+t}(\mu^{\circ \vee})
\end{align*}

as required.
\end{proof}

\section{Multiplicities in the category 
\texorpdfstring{$\cV$}{V}}\label{sec:mult}
In this section we multiplicities in standard filtrations and Jordan-Holder 
filtrations in category $\cV$. 
Recall that the objects $ \mathbf{T}(\lam)$ in $\cV$ have a standard filtration 
by objects $\mathbf{V}(\mu)$, which is induced from the standardly-filtered 
objects $V_t^{\otimes \abs{\lam^{\bullet}}} \otimes V_t^{* \otimes 
\abs{\lam^{\circ}}}$ (see \cite{EHS} for details). 
\begin{theorem}\label{thrm:equiv_multip}
 Let $t \in \bC$, and let $\lambda, \mu$ be two bipartitions. Then $$\left( 
\mathbf{T}(\lam):\mathbf{V}(\mu) \right)_{\cV} = \left[ \mathbf{V}(\lam) 
:\mathbf{L}(\mu) \right]_{\cV}$$
 where the LHS denotes the multiplicity of the standard object $\mathbf{V}(\mu)$ 
in the standard filtration of the standard object $\mathbf{T}(\lam)$, and the 
RHS denotes the multiplicity of the simple object $\mathbf{L}(\lam)$ among the 
Jordan-Holder components of $\mathbf{V}(\mu)$.
 
 Moreover, these multiplicities coincide with the multiplicity 
$D^{\lam}_{\mu}(t)$ (see Section \ref{sssec:formal_param} for definition, and 
\cite[Section 6]{CW}).
\end{theorem}
\begin{proof}
 Denote: $$K^{\lam}_{\mu}(t): = \left( \mathbf{T}(\lam):\mathbf{V}(\mu) 
\right)_{\cV}  , \; \; \; \tilde{K}^{\lam}_{\mu}(t) = \left[ \mathbf{V}(\lam) 
:\mathbf{L}(\mu) \right]_{\cV}$$
 The multiplicities $\tilde{K}^{\lam}_{\mu}(t)$ can be shown to be either $1$ 
or 
$0$, and can be computed using cap diagrams (cf. original definitions in 
\cite{BS}, \cite{MS}). We will use the cap diagrams as defined in \cite[Section 
4.5, Section 8.3]{EHS}. The {\it cap diagram} corresponding to $d'_{\lam}$ is 
the diagram $d'_{\lam}$ together with additional arcs (``caps'') which have a 
\InnaC{left} end at $\times$ and a right end at $\bigcirc$. These caps are 
required to satisfy the following conditions:
 \begin{enumerate}
 \item The caps should not intersect each other.
  \item Any $\bigcirc$ which is inside a cap should be the right end of some 
other cap.
  \item Any $\times$ should be the left end of exactly one cap.
 \end{enumerate}
 
 In \cite[Lemma 8.3.2]{EHS} it is shown that $\tilde{K}^{\lam}_{\mu}(t)=1$ iff 
one can obtain $d'_{\lam}$ from $d'_{\mu}$ by 
moving finitely many crosses in the cap diagram of $d'_{\mu}$ 
from the left end of a cap to the right end of this cap.

Let us give a 
sketch of the proof: using Lemma \ref{lem:V_lam_highest_weight}, one reduces 
the problem to the computation of Kazhdan-Lusztig coefficients in the 
$\gl(m|n)$-module $\mathit{K}(\lam)$. A simple module $\mathit{L}(\mu)$ sits in 
$\mathit{K}(\lam)$ precisely if one can obtain the diagram $d_{\lam^{\vee}}$ 
from $d_{\mu^{\vee}}$ by moving finitely many crosses from the left end of a 
cap 
to the right end of this cap (cf. \cite{MS}). By Lemma \ref{lem:CW_vs_V}, we 
have $d'_{\lam} = d_{\lam^{\vee}}$, which proves the required statement.
 
 \begin{remark}
 \InnaC{The transposition of the second Young diagram in the 
bipartitions appearing above comes from the different choice of simple 
roots in \cite{EHS}.}
\end{remark}
 
 In \cite[Section 6]{CW}, it is shown that $lift_t(\mathbf{T}(\lam)) = 
\bigoplus_{\mu} \mathbf{T}(\mu)^{\oplus D^{\lam}_{\mu}(t)}$ (see Section 
\ref{sssec:formal_param} for definition of $lift_t$), where $D^{\lam}_{\mu}(t)$ 
is $1$ or $0$, and it can be computed explicitly using cap diagrams: namely, 
$D^{\lam}_{\mu}(t) =1$ iff $d'_{\lam}$ can be obtained from $d'_{\mu}$ by moving 
finitely many $\times$ in the cap diagram of $d'_{\mu}$ from the left end of a 
cap to the right end of this cap\footnote{Note that the caps we describe are 
``complimentary'' to the caps used in \cite{CW}: the latter would have $\times$ 
on the right end and $\bigcirc$ on the left end, which means that after 
switching the symbols they would become our caps.}. 
 
Hence $$\tilde{K}^{ \lam}_{\mu}(t)= D^{\lam}_{\mu}(t).$$

 Now, consider the object $Y(\lam):= S^{\lam^{\bullet}} V_t \otimes 
S^{\lam^{\circ}} V_t^*$ in $\mathcal{D}_t$. We denote by $b^{\lam}_{\mu}(t)$ the 
multiplicities in the decomposition $$Y(\lam) = \bigoplus_{\mu} 
\mathbf{T}(\mu)^{\oplus b^{\lam}_{\mu}(t)}$$
 
 We consider the objects $\mathbf{T}(\lam)$, $Y(\lam):= S^{\lam^{\bullet}} V_T 
\otimes S^{\lam^{\circ}} V_T^*$ in $\mathcal{D}_T$, where $T$ is a formal 
parameter. We will denote by $B^{\lam}_{\mu}$ the corresponding multiplicity in 
this case.
 
 One can immediately see from the definition of $lift_t$ that $lift_t(Y(\lam)) = 
Y(\lam)$ in $\mathcal{D}_T$, and therefore 
 $$ \bigoplus_{\mu} lift_t(\mathbf{T}(\mu))^{\oplus b^{\lam}_{\mu}(t)} = 
\bigoplus_{ \mu'} \mathbf{T}(\mu')^{\oplus \sum_\mu b^{\lam}_{\mu}(t) 
D^{\mu}_{\mu'}(t)} = lift_t(Y(\lam)) =Y(\lam) =\bigoplus_{\mu'} 
\mathbf{T}(\mu')^{\oplus B^{\lam}_{\mu'}}  $$
 
 Considering $b, B, D$ as (infinite) matrices whose entries are numbered by 
pairs of bipartitions, we obtain: $B = b(t) D(t)$.
 
 Next, it was shown in \cite[Corollary 7.1.2]{CW} that $$B^{\lam}_{\mu} = 
\sum_{\kappa \in \mathcal{P}} LR^{\lam^{\bullet}}_{\mu^{\bullet}, \kappa} 
LR^{\lam^{\circ}}_{ \mu^{\circ},\kappa}$$
 where $LR$ denote the Littlewood-Richardson coefficients, and the sum is over 
all the partitions $\kappa$.
 
 On the other hand, recall that $Y(\lam) \subset V_t^{\otimes 
\abs{\lam^{\bullet}}} \otimes V_t^{* \otimes \abs{\lam^{\circ}}}$ is a 
standardly-filtered object in $\cV$, and the multiplicities of standards in this 
filtration are known, due to \cite[Lemma 8.1.4]{EHS}, \cite{PS}:
 $$\left( Y(\lam): \mathbf{V}(\mu) \right) = \sum_{\kappa \in \mathcal{P}} 
LR^{\lam^{\bullet}}_{\mu^{\bullet}, \kappa} LR^{\lam^{\circ}}_{ 
\mu^{\circ},\kappa}$$
 
 Thus $( Y(\lam): \mathbf{V}(\mu))$ is the $(\lam, \mu)$ entry in $b(t) D(t)$.
 
 Now, $$\left( Y(\lam): \mathbf{V}(\mu) \right) = \sum_{\mu'} 
[Y(\lam):\mathbf{T}(\mu')] \left( \mathbf{T}(\mu'):\mathbf{V}(\mu) \right) = 
\sum_{\mu'} b^{\lam}_{\mu'}(t) K^{\mu'}_{\mu}(t)$$
 
 Hence $$b(t)D(t) = b(t) K(t)$$ We claim that the matrix $b(t)$ is invertible. 
Indeed, we can order the set of bipartitions by total size ($\abs{\lam} := 
\abs{\lam^{\bullet} } + \abs{\lam^{\circ}}$); bipartitions of the same size can 
be ordered arbitrarily.
 Then it is easy to see that $B, D(t)$ are matrices which are lower-triangular 
(with finitely many non-zero entries in each column), and all entries on the 
diagonal are $1$ (cf. \cite[Section 6]{CW}). Thus $B, D(t)$ are invertible, and 
so is $b(t)$. We conclude that  
 $$D(t) = K(t)$$ which, together with the fact that $\tilde{K}^{ \lam}_{\mu}(t)= 
D^{\lam}_{\mu}(t)$, completes the proof.

\end{proof}

\begin{example}
 Let $t=0$, $\lam = (\InnaD{\young(\bullet)}, \square)$, $\mu = 
(\InnaC{\varnothing}, 
\InnaC{\varnothing})$. Then $\mathbf{T}(\InnaD{\young(\bullet)}, \square) = V_t 
\otimes 
V_t^*$, $\mathbf{T}(\InnaC{\varnothing}, \InnaC{\varnothing}) = 
\mathbf{V}(\InnaC{\varnothing}, \InnaC{\varnothing}) = 
\mathbf{L}(\InnaC{\varnothing}, \InnaC{\varnothing}) = \triv$. Then $$\left( 
\mathbf{T}(\InnaD{\young(\bullet)}, \square):\mathbf{V}(\InnaC{\varnothing}, 
\InnaC{\varnothing}) \right)_{\cV} = \left( \mathbf{V}(\InnaD{\young(\bullet)}, 
\square):\triv \right)_{\cV} = 1$$ while $$\left( 
\mathbf{T}(\InnaC{\varnothing}, 
\InnaC{\varnothing}):\mathbf{V}(\InnaD{\young(\bullet)}, 
\square) \right)_{\cV} = \left( 
\mathbf{V}(\InnaC{\varnothing},\InnaC{\varnothing}):\mathbf{L}(\InnaD{
\young(\bullet)}, 
\square) \right)_{\cV} = 0$$  
\end{example}

As a corollary, we obtain the following formula for the dimension of the 
Hom-space for tilting modules:
\begin{corollary}\label{cor:Hom_two_tilting}
 Let $\lam, \mu$ be two bipartitions. Then $$\dim \Hom_{\cV} (\mathbf{T}(\lam), 
\mathbf{T}(\mu)) = \sum_{\nu \in \mathcal{P} \times \mathcal{P}} \left( 
\mathbf{T}(\lam):\mathbf{V}(\nu) \right)_{\cV} \left( 
\mathbf{T}(\mu):\mathbf{V}(\nu) \right)_{\cV} $$
 
\end{corollary}

\begin{remark}
 In any highest weight category, we have
 $$\dim \Hom_{\cV} ({T}_1, {T}_2) = \sum_{\nu} \left( {T}_1:\Delta(\nu) 
\right)_{\cV} \left( {T}_2:\nabla(\nu) \right)_{\cV} $$
 where $T_1, T_2$ are tilting objects, and $\Delta(\nu)$, $\nabla(\nu)$ are the 
standard and costandard objects corresponding to weight $\nu$.
 
 In our case, in the (local) highest-weight subcategories $\mathcal{V}^k_t$ 
there is an exact duality functor taking standard objects to co-standard 
objects; this makes Corollary \ref{cor:Hom_two_tilting} an immediate consequence 
of the above formula.
\end{remark}

\begin{proof}
 It was proved in \cite[Section 6]{CW} that $$\dim \Hom_{\cV} (\mathbf{T}(\lam), 
\mathbf{T}(\mu)) = \sum_{\nu} D^{\lam}_{\nu}(t) D^{\mu}_{\nu}(t)$$
 We have just proved that $D^{\lam}_{\nu}(t) =\left( 
\mathbf{T}(\lam):\mathbf{V}(\nu) \right)_{\cV}$ (this value is either $0$ or 
$1$), and the statement follows.
\end{proof}

\section{Categorical \texorpdfstring{$\sll_{\bZ}$}{sl}-action in tensor 
categories}\label{sec:skeletal_action}
The notions of categorical actions which we will use will be based on the 
definitions in \cite{CR, R, BLW}. We will also define the notion of a 
categorical action of $\sll_{\bZ}$ corresponding to a rigid SM category, in the 
sense of \cite[Section 5.1.1]{R} (witout the requirement that the category be 
abelian, nor that the complexified (split) Grothendieck group gives an integral 
representation of $\sll_{\bZ}$).
\subsection{Categorical type A action}\label{ssec:skeletal_action_def}
We start with the following definition, which is a slightly stronger form of the 
definition of a type A action given by Rouquier  in \cite[Section 
5.1.1]{R}\footnote{In \cite{R}, the first condition is replaced by a requirement 
that natural transformations $E_aF_a \to F_aE_a$ defined through $\rx, \tau$ 
would be invertible.}:
\begin{definition}\label{def:skeletal_action_sl_2}
 Let $\mathcal{A}$ be an additive category. A categorical type A action on 
$\mathcal{A}$ consists of the data $(F, E, \rx, \tau)$, where $(E, F)$ are an 
adjoint pair of (additive) endofunctors of $\mathcal{A}$, $\rx \in \End(F)$, 
$\tau \in \End(F^2)$ and these satisfy:
 \begin{itemize}
  \item $F$ is isomorphic to the left adjoint of $E$,
  \item For any $d \geq 2$, the natural transformations $\rx, \tau$ define an 
action of the degenerate affine Hecke algebra on $F^d$ by \begin{align*}
 \mathrm{dAHA}_d = \bC[\rx_1, \ldots, \rx_d] \otimes \bC[S_d] &\longrightarrow 
\End(F^d)\\
 \rx_i &\mapsto F^{d-i} \rx F^{i-1} \\
 (i, i+1) &\mapsto F^{d-i-1} \tau F^{i-1}
\end{align*}

 \end{itemize}

\end{definition}

The definition of a functor between categorical $\sll_{\bZ}$-modules is the same 
as in \cite[5.2.1]{CR}.

In the spirit of \cite{R, BLW}, if $\mathcal{A}$ is Karoubian, we can consider 
the generalized eigenspaces of the operator $\rx$. This gives us a decomposition 
$F = \bigoplus_{a\in \bC} F_a$, where $F_a$ is the generalized eigenspace of 
$\rx$ corresponding to eigenvalue $a$. Similarly, one can define a decomposition 
$E = \bigoplus_{a\in \bC} E_a$ where $E_a$ is adjoint to $F_a$.

Notice that this does not necessarily give a categorical $\sll_{\bZ}$-action on 
$\mathcal{A}$ as described in \cite{R, BLW}, since the eigenvalues of $\rx$ are 
not necessarily integers.

\subsection{Categorical type A action in 
\texorpdfstring{$\Repc(\gl(V))$}{Rep(gl(V))}}\label{
ssec:skeletal_action_tens_cat}

Let $\mathcal{C}$ be a rigid SM category, and fix $V \in \mathcal{C}$. Consider 
the Lie algebra object $\gl(V)$ and the category $\Repc(\gl(V))$ of 
representations of $\gl(V)$ in $\mathcal{C}$.

We construct a categorical type A action on $\Repc(\gl(V))$ as follows:
\begin{construction}\label{constr:skeletal_action_Rep_gl}
 \mbox{}
 
\begin{itemize}
 \item Define the functors $F:= V \otimes (-)$, $E:= V^* \otimes (-)$.
 \item Define the natural transformation $\tau:= \sigma_{V, V}\otimes \id \in 
\End(F^2)$.
 \item Define the natural transformation $\rx \in \End(F)$ by: $\rx_{(M, \rho)} 
\in \End(V\otimes M)$ corresponds to $\rho$ under the isomorphism $$ 
\End(V\otimes M) \cong \Hom_{\C}(V \otimes V^* \otimes M, M)$$
\end{itemize}

\end{construction}

By definition of rigidity, the functors $E, F$ are biadjoint. The action of the 
degenerate affine Hecke algebra acts on $F^d$ follows from the axioms of the 
symmetry morphism $\sigma_{-, -}$, together with the following lemma:

\begin{lemma}
 The natural transformations $\rx, \tau$ satisfy: $\tau\circ \rx F - F\rx \circ 
\tau = \id$, $\tau \circ F \rx - \rx F  \circ \tau = -\id$.
\end{lemma}
\begin{proof}
 The proof is a straightforward computation, relying on the fact that for any 
$(M, \rho)$, $$\rx_{V \otimes (M, \rho)} = \sigma_{V, V} \otimes \id_M + 
(\sigma_{V, V}  \otimes \id_M) \circ \left(\id_V \otimes \rx_{(M, \rho)} \right) 
\circ (\sigma_{V, V}  \otimes \id_M)  $$
 This implies that $\rx F = \tau + \tau \circ F \rx \circ \tau$, and the 
statement follows.
\end{proof}

\begin{remark}
 As in \cite[Section 7.4]{CR}, one can alternatively define $x$ using the 
Casimir natural transformation of the identity functor on $\Repc(\gl(V))$. The 
Casimir $C \in \End(\id_{\Repc(\gl(V))})$ is defined as follows: for every $(M, 
\rho) \in \Repc(\gl(V))$, $C_{(M, \rho)}$ is the composition $$  M 
\xrightarrow{coev_{\gl(V)} \otimes \id_M} \gl(V) \otimes \gl(V) \otimes M 
\xrightarrow{\id_{\gl(V)} \otimes \rho} \gl(V) \otimes M \xrightarrow{\rho} M $$
 We then have $$\rx_M = \frac{1}{2} \left( C_{V\otimes M} - C_V \otimes \id_M - 
\id_V \otimes C_M \right)$$
\end{remark}

\mbox{}

Here is an example which illustrates the effect of the natural transformation 
$\rx$ in $\Repc(\gl(V))$:
\begin{example}
Assume that $\C$ is Karoubian (all idempotents split). Consider the 
decomposition 
 $$F(V^{\otimes n-1}) = V^{\otimes n} \cong \bigotimes_{\mu \vdash n} S^{\mu} V 
\otimes \mu$$ 
 The action of $\rx_{V^{\otimes n-1}}$ is given by a collection of endomorphisms 
$\phi_{\mu} \in \End(\mu)$, such that $\phi_{\mu}$ commutes with the action of 
$G:=S_{\{2, 3, \ldots, n-1\}} \subset S_n$. Thus $\phi_{\mu}$ is diagonalizable 
with eigenspaces given by irreducible $G$-summands of $\mu$. Identifying $G 
\cong S_{n-1}$, we get
 $$ Res^{S_n}_G(\mu) = \bigoplus_{\lambda \in \mu^-} \lambda$$ where the sum is 
over Young diagrams obtained from $\mu$ by erasing one box. The eigenvalue 
corresponding to eigenspace $\lambda$ is then $ct(\mu - \lambda)$.
 
 Now, $$F(S^{\lambda} V) = V\otimes S^{\lambda} V \cong \bigoplus_{\mu \in 
\lambda + \square} S^{\mu} V$$
In fact, we have: 
 $$F(S^{\lambda} V \otimes \lambda) \cong \bigoplus_{\mu \in \lambda + \square} 
S^{\mu} V \otimes \lambda \subset  \bigoplus_{\mu \in \lambda + \square} S^{\mu} 
V \otimes \mu$$
 Thus $\rx_{S^{\lambda} V}$ acts on each direct summand $S^{\mu} \subset 
F(S^{\lambda} V )$ ($\mu \in \lambda + \square$) by the scalar $ct(\mu 
-\lambda)$.
 
\end{example}

From now on we let $\C$ be a tensor category (in particular, abelian), and fix 
$V \in \C$. The category $\Repc(\gl(V))$ is again a tensor category.

We will now consider the full tensor subcategory $\Repc(GL(V), \eps)$ of 
$\Repc(\gl(V))$ (see \cite{D2}, \cite[Section 11]{EHS} for definition). This 
subcategory is defined as the category of representations in $\C$ of a certain 
Hopf algebra object, which is the analogue of $\mathcal{O}(GL(W))$ for a vector 
space $W$. In particular, any subquotient of a direct sum of mixed tensor powers 
of $V$ has a natural structure making it an object of $\Repc(\gl(V))$, and it 
can be shown that the objects of $\Repc(GL(V), \eps)$ are all of this form (cf. 
\cite[Appendix]{D2}).

The subcategory $\Repc(GL(V), \eps)$ is preserved by the type A action defined 
above.

\begin{example}\label{ex:Rep_GL_V}
\mbox{}

 \begin{enumerate}[label={\it \alph*)}]
  \item Let $(\C, V):= (SVec, \bC^{m|n})$. In this case $\Repc(GL(V), \eps) = 
Rep(\gl(m|n))$\footnote{This is ``half'' of the category of all 
finite-dimensional representations of $\gl(m|n)$, as explained in Section 
\ref{ssec:notn_rep_superalg}.}, and the induced type A action on $Rep(\gl(m|n))$ 
coincides with $\sll_{\bZ}$ action as in \cite{BLW} (see Section 
\ref{ssec:cat_action_gl_m_n}). 
  \item Let $(\C, V):= (\cV, V_t)$ ($t \in \bC$). In this case $\Repc(GL(V), 
\eps) \cong \cV$ (see \cite[Section 11]{EHS}), which defines an type A action 
on 
$\cV$, studied below.
 \end{enumerate}
 
\end{example}
We will be interested in the categories $\Repc(GL(V), \eps)$ since the above 
examples are universal among the pairs $(\Repc(GL(V), \eps), V)$ (see 
\cite{EHS}).

The following lemma is straightforward (cf. \cite[Sect. 7]{D1}, \cite[Section 
11]{EHS}).
\begin{lemma}\label{lem:equivar_skeletal_act}
 Let $(\C, V)$, $(\C', V')$ be a pair of tensor categories with a fixed object. 
Consider their respective type A actions, and let $G: \C \longrightarrow \C'$ be 
an SM functor such that $V \mapsto V'$. Then $G$ induces a equivariant SM 
functor $G:\Repc(\gl(V)) \longrightarrow Rep_{\C'}(\gl(V'))$, and a equivariant 
SM functor $G:\Repc(GL(V), \eps) \longrightarrow Rep_{\C'}(GL(V'), \eps)$.
\end{lemma}

\begin{corollary}\label{cor:equiv_Deligne_functor}
Let $(\C, V)$ be a tensor category with a fixed object, and set $t := \dim V$. 
The additive SM functors
$$\mathcal{D}_t \stackrel{\mathbf{S}_{V}}{\longrightarrow} \Repc(GL(V), \eps)$$ 
are equivariant with respect to the type A action.

\end{corollary}

As it was mentioned before, if the eigenvalues of the operator $\rx$ on 
$\Repc(GL(V), \eps)$ are integers, then the functors $F_a, E_a$ define an 
$\sll_{\bZ}$-action on the Grothendieck group of the abelian category 
$\Repc(GL(V), \eps)$, in the sense of \cite{R, BLW}. 

In the Example \ref{ex:Rep_GL_V} (a), this is indeed the case. We will show that 
 
\begin{proposition}\label{prop:eigenval_V_t_integral}
  The eigenvalues of $\rx$ in $\cV$ are integers iff $t \in \bZ$. 
\end{proposition}
\begin{proof}
 It will be shown in Sections \ref{sec:action_on_GL_t} and 
\ref{sec:actions_on_V} that the set of eigenvalues of $\rx$ is $\bZ \cup \bZ-t$.
\end{proof}

Thus the type A action on $\cV$ defines an $\sll_{\bZ}$-action whenever $t \in 
\bZ$.

We now formulate a corollary, which is a direct consequence of \cite[Theorem 
11.1.2]{EHS}. This corollary essentially states that any categorical type A 
action 
on a category of the form $\Repc(GL(V), \eps)$ which originates in an type A 
action can be described using the categorical type A actions on the 
tensor categories $\cV$, $Rep(\gl(m|n))$. 

Recall that tensor categories which are finite-length abelian categories and 
have 
finite-dimensional $\Hom$-spaces are called pre-Tannakian (see \cite[Section 
8.1]{D1} for 
definition).

\begin{corollary}
 Let $\C$ be a pre-Tannakian tensor category, and let $V \in \C$ whose dimension 
$t:=\dim(V)$ is an integer.  
 If $V$ is ``torsion-free'', i.e. $S^{\lam} V \neq 0$ for any $\lam$, then 
there 
exists an equivariant equivalence of $\sll_{\bZ}$-categorical modules 
$$\Repc(GL(V), \eps) \cong \cV$$
 If $V$ is ``torsion'', i.e. there exists $\lam$ such that $S^{\lam} V=0$, then 
there exists an equivariant equivalence of $\sll_{\bZ}$-categorical modules 
$$\Repc(GL(V), \eps) \cong Rep(\gl(m|n))$$ for some $m, n\in \bZ_+$, $m-n=t$.
 \end{corollary}
 
 \begin{remark}
  A similar result holds when $t:=\dim(V)$ is not an integer, but as it is seen 
in Theorem \ref{thrm:action_Deligne_cat_ss}, this is an equivariant equivalence 
of $\sll_{\bZ}\times \sll_{\bZ}$-categorical modules.
 \end{remark}

 \subsection{Categorical action on 
\texorpdfstring{$Rep(\gl(m|n))$}{Rep(gl(m|n))}}\label{ssec:cat_action_gl_m_n}
 
We now describe two special cases of the above construction: a categorical 
$\sll_{\bZ}$-action on $Rep(\gl(m|n))$ and specifically on $Rep(\gl_m)$ (for 
explicit descriptions of these actions, see e.g. \cite{CR, BLW}). 
\begin{example}
{\bf Case of $Rep(\gl_m)$}: Recall that the singular vectors in 
$\mathit{L}(\lam) \otimes \bC^m$ have weights $\vec{\lam}+\eps_i$ for different 
$i$; in terms of bipartitions, it means that $$F(\mathit{L}(\lam)) \cong 
\bigoplus_{\mu \in \lam + \InnaD{\young(\bullet)} \sqcup \lam - \square} 
\mathit{L}(\mu) $$ 
Moreover, recalling the definition of $\rx$ through the Casimir operator, one 
can immediately see that the eigenvalue of $\rx_{\mathit{L}(\lam)}$ on 
$\mathit{L}(\mu)$ is $\vec{\lam}_i -i +1$, where $i$ is such that $\vec{\mu} = 
\vec{\lam} +\eps_i$. Again, in terms of bipartitions, this value is equal to 
$ct(\mu^{\bullet} -\lam^{\bullet})$ if $\mu\in \lam + \InnaD{\young(\bullet)}$, 
and is 
equal to $-ct(\lambda^{\circ} -\mu^{\circ} ) -m$ if $\mu\in \lam - \square$.

Thus $$F_a(\mathit{L}(\lam)) = \bigoplus_{\mu\in \lam + 
\InnaD{\young(\bullet)}_a } 
\mathit{L}(\mu)  \oplus  \bigoplus_{\mu\in \lam - \square_{-m-a} 
}\mathit{L}(\mu) $$
(notation as in Section \ref{ssec:notn_Young}), and there is an isomorphism of 
$\sll_{\bZ}$-modules
$$\bC \otimes_{\bZ} Gr(Rep(\gl_m)) \cong \wedge^m \bC^{\bZ}$$
 \end{example}
 
 A similar situation appears in the general case of $Rep(\gl(m|n))$ (cf. proof 
of Proposition \ref{prop:action_standards}: it was shown in \cite{BLW} that in 
this case 
 there is an isomorphism of $\sll_{\bZ}$-modules
$$\bC \otimes_{\bZ} Gr(Rep(\gl(m|n))) \stackrel{\sim{}}{\longrightarrow} 
\wedge^m \bC^{\bZ} \otimes \wedge^n (\bC^{\bZ})^*$$ sending Kac modules to pure 
wedges.

\section{Operator \texorpdfstring{$\rx$}{x}}\label{sec:action_on_GL_t}

\subsection{Eigenvalues of the operator 
\texorpdfstring{$\rx$}{x}}\label{ssec:eigenval_oper_x}
We will now describe the generalized eigenspaces of the natural transformation 
$\rx$ when $t \notin \bZ$.
\begin{proposition}\label{prop:eigenvalues_x_Deligne}
 Let $\lam$ be a bipartition, and let $t \notin \bZ$. The generalized eigenspace 
of $\rx_{\mathbf{T}(\lam)}$ corresponding to $a\in \bC$ is 
 $$\bigoplus_{\mu\in \lam + \InnaD{\young(\bullet)}_a } \mathbf{T}(\mu) \oplus  
\bigoplus_{\mu\in \lam - \square_{-\InnaC{(a+t)}} } \mathbf{T}(\mu)$$
 (notation as in Section \ref{ssec:notn_Young}).
\end{proposition}
\begin{remark}
 Note that since $t \notin \bZ$ and thus $Rep(GL_t)$ is semisimple, the operator 
$\rx_{\mathbf{T}(\lam)}$ is diagonalizable.
\end{remark}

\begin{proof}
Fix $\lam \vdash (r, s)$, and let $t \in \bC \setminus \{0, \pm 1, \pm 2, 
\ldots, \pm (r+s+1) \} $.

The statement of the proposition is equivalent to computing the generalized 
eigenvalues of $\rx_{\mathbf{T}(\lam)}$. Each summand $\mathbf{T}(\mu)$ of 
$F(\mathbf{T}(\lam))$ is indecomposable, and hence corresponds to a generalized 
eigenvalue of $\rx_{\mathbf{T}(\lam)}$.

We would like to say that the eigenvalue of $\rx_{\mathbf{T}(\lam)}$ on 
$\mathbf{T}(\mu)$ depends polynomially on $t$; we will compute separately the 
eigenvalue in the special case of $t \in \bZ_{>>0}$, and the polynomiality would 
allow us to interpolate this result to all values of $t$. 

 We will denote $x^{\lam}_{\mu, t} = p \circ \rx_{\mathbf{T}(\lam)} \circ i$ 
where $i, p$ are the inclusion and the projection to of the direct summand 
$\mathbf{T}(\mu) \subset \mathbf{T}(\lambda) \otimes V_t$ in $\mathcal{D}_t$.

\mbox{}

{\bf Case $t=d \in \bZ_{> r+s}$}: 

In this case $\mathbf{S}_{t=d}(\mathbf{T}(\mu)) = \mathit{L}(\mu)$ (see Notation 
\ref{ssec:notn_rep_superalg}). Then 
$$\mathbf{S}_{t=d}(F(\mathbf{T}(\lam))) = F(\mathit{L}(\lam)) \cong 
\bigoplus_{\mu \in \lam + \InnaD{\young(\bullet)} \sqcup \lam - \square} 
\mathit{L}(\mu) 
$$ 

By Section \ref{ssec:cat_action_gl_m_n}, the generalized eigenvalue of 
$\rx_{\mathit{L}(\lam)}$ on $\mathit{L}(\mu)$ is $ct(\mu^{\bullet} 
-\lam^{\bullet})$ if $\mu\in \lam + \InnaD{\young(\bullet)}$, and is equal to 
$-ct(\lambda^{\circ} -\mu^{\circ} ) -d$ if $\mu\in \lam - \square$. By Corollary 
\ref{cor:equiv_Deligne_functor}, the generalized eigenvalue of $ 
\rx_{\mathbf{T}(\lam)}$ on $\mathbf{T}(\mu)$ is the same.

\mbox{}

We will now use the map $Lift_t: K_0(\mathcal{D}_t) \rightarrow 
K_0(\mathcal{D}_T)$ defined in Section \ref{sssec:formal_param}. As it was 
mentioned before, for any $t \notin \{0, \pm 1, \pm 2, \ldots, \pm 
(\abs{\mu^{\bullet}} + \abs{\mu^{\circ}}) \}$, we have: $Lift_t(\mathbf{T}(\mu)) 
=\mathbf{T}(\mu)$. 

Consider the idempotent $e^{\lam}_{\mu} \in Br_{\bC}(r+1, s)$ which is the 
projector onto the multiplicity space of $T^{\mu}$ in $V\otimes 
\mathbf{T}(\lam)\subset V_t^{\otimes r+1} \otimes V_t^{* \otimes s}$ (by abuse 
of notation, we will denote by $e^{\lam}_{\mu}$ the corresponding idempotent in 
$Br_{\bC\InnaD{((T-t))}}(r+1, s)$ as well)\footnote{We stress that these 
idempotents 
depend on $\lambda$.}.

In $\mathcal{D}_T$, we have (see also \cite[Section 7]{CW}):
$$ V_T \otimes \mathbf{T}(\lam) \cong \bigoplus_{\mu\in \lam + 
\InnaD{\young(\bullet)}} 
\mathbf{T}(\mu) \oplus \bigoplus_{\mu\in \lam - \square} \mathbf{T}(\mu)$$
Together these two facts imply that for $t \in \bC \setminus \{0, \pm 1, \pm 2, 
\ldots, \pm (r+s+1) \}$, a similar decomposition holds in $\mathcal{D}_t$:
$$F(\mathbf{T}(\lam)) = V_t \otimes \mathbf{T}(\lam) \cong \bigoplus_{\mu\in 
\lam + \InnaD{\young(\bullet)}} \mathbf{T}(\mu) \oplus \bigoplus_{\mu\in \lam - 
\square} 
\mathbf{T}(\mu)$$
Thus the idempotents $e^{\lam}_{\mu}$ are primitive in $\mathcal{D}_T$ and for 
the above values of $t$, and the lifting of $e^{\lam}_{\mu}$ is exactly 
$e^{\lam}_{\mu} \in Br_{\bC\InnaD{((T-t))}}(r+1, s)$ whenever $t \in \bC 
\setminus \{0, 
\pm 1, \pm 2, \ldots, \pm (r+s+1) \} $.

\mbox{}

{\bf Computation of $x^{\lam}_{\mu, t}$ for all $t \in \bC \setminus \{0, \pm 1, 
\pm 2, \ldots, \pm (r+s+1) \}$:}
We will now describe how to lift the endomorphism $x^{\lam}_{\mu, t}$ to an 
endomorphism in $\mathcal{D}_T$.

We have $$x^{\lam}_{\mu, t} = e^{\lam}_{\mu} \rx_{_{V_t^{\otimes r} \otimes 
V_t^{* \otimes s}}} e^{\lam}_{\mu}$$  as an endomorphism in 
$\End_{\mathcal{D}_t}(\mathbf{T}(\mu)) \subset \End_{\mathcal{D}_t}(V_t^{\otimes 
{r+1}} \otimes V_t^{* \otimes s})$.
First, consider the morphism $ \rx_{_{V_T^{\otimes r} \otimes V_T^{* \otimes 
s}}}$ in $\mathcal{D}_T$, where $T$ is a formal variable. We have $$ 
\rx_{_{V_T^{\otimes r} \otimes V_T^{* \otimes s}}} \rvert_{T=t} = 
\rx_{_{V_t^{\otimes r} \otimes V_t^{* \otimes s}}}$$ (this is obvious, since the 
parameter $t$ is not involved in the definition of $\rx_{_{V_t^{\otimes r} 
\otimes V_t^{* \otimes s}}}$).

Secondly, we define an endomorphism $$x^{\lam}_{\mu, T}:=e^{\lam}_{\mu} 
\rx_{_{V_T^{\otimes r} \otimes V_T^{* \otimes s}}} e^{\lam}_{\mu} \in 
\End_{\mathcal{D}_T}(\mathbf{T}(\mu)) \subset \End_{\mathcal{D}_T}(V_T^{\otimes 
{r+1}} \otimes V_T^{* \otimes s})$$ This endomorphism does not depend on $t$. 
Since $\mathbf{T}(\mu)$ is simple, this is in fact a scalar multiple of 
$e^{\lam}_{\mu}$; we will denote the corresponding scalar (a formal Laurent 
series in $T$) by $\chi^{\lam}_{\mu}$. The previous paragraph implies that the 
endomorphism $x^{\lam}_{\mu, T}$ is the lift of $x^{\lam}_{\mu, t}$ for $t 
\notin \{0, \pm 1, \pm 2, \ldots, \pm (r+s+1) \}$.

This immediately implies that $x^{\lam}_{\mu, t} = \chi^{\lam}_{\mu}$ for any 
$t$ as above.  

Applying this to the case $t=d\in \bZ_{>>0}$, we conclude that 
$$\chi^{\lam}_{\mu} = \begin{cases} ct(\mu^{\bullet} -\lam^{\bullet}) &\text{  
if  } \mu\in \lam + \InnaD{\young(\bullet)} \\ -ct(\lambda^{\circ} -\mu^{\circ} 
) -T   
&\text{  if  } \mu\in \lam - \square \\ 0 &\text{  else} \end{cases}$$

\end{proof}

\section{Categorical \texorpdfstring{$\sll_{\bZ}$-actions}{actions} on 
\texorpdfstring{$\mathcal{D}_t$}{Deligne 
categories}}\label{sec:categorical_sl_infty_actions_GL_t}
We will now define a categorical $\sll_{\bZ}$-action on the category 
$\mathcal{D}_t$.
\begin{definition}
 Let $a \in \bC$. We define the endofunctor $F_a$ of $\mathcal{D}_t$ so that 
$F_a(M)$ is the generalized $a$-eigenspace of $\rx_M$. 
\end{definition}
We have: $$F = \bigoplus_{a \in \bC} F_a$$
The endofunctor $E_a$ can be defined similarly as a direct summand of $E$; it is 
the left adjoint to $F_a$ and isomorphic to the right adjoint of $F_a$.

 From now and until the end of this section, assume that $t \notin \bZ$. The 
category $\mathcal{D}_t$ is then semisimple, and on simple objects, the actions 
of $E_a, F_a$ are given by Proposition \ref{prop:eigenvalues_x_Deligne}:
 \begin{align}\label{eq:F_a_action}
  &F_a(\mathbf{T}(\lambda)) = \bigoplus_{\mu\in \lam + \InnaD{\young(\bullet)}_a 
} 
\mathbf{T}(\mu) \oplus  \bigoplus_{\mu\in \lam - \square_{-\InnaC{(a+t)}} } 
\mathbf{T}(\mu)\\
 &E_a(\mathbf{T}(\lambda)) = \bigoplus_{\mu\in \lam - 
\InnaD{\young(\bullet)}_{a}} 
\mathbf{T}(\mu) \oplus \bigoplus_{\mu\in \lam + \square_{-\InnaC{(a+t)}}} 
\mathbf{T}(\mu)
 \end{align}

 One immediately observes that $F_a, E_a = 0$ when $a \notin \bZ \cup \bZ-t$.

 In this case, the union $\bZ \cup \bZ-t$ is disjoint, which implies that there 
are two separate copies of $\sll_{\bZ}$ acting on $\mathcal{D}_t$. Indeed, 
when $t \notin \bZ$, one can define two commuting natural transformations $\rx', 
\rx'' \in \End(F)$ such that $\rx = \rx' +\rx''$, and $\rx'$ has eigenvalues in 
$ \bZ$, while $\rx''$ has eigenvalues in $\bZ-t$. 
 
 The construction of Rouquier (see \cite{R}, \cite[7.4]{CR}) then gives an 
action of $\sll_{\bZ}\times \sll_{\bZ}$ on $\mathcal{D}_t$, where the first 
$\sll_{\bZ}$-copy corresponds to the data $(F, E, \rx', \tau)$, while the second 
$\sll_{\bZ}$-copy corresponds to the data $(F, E, \rx''+t\id, \tau)$. 
 
 In these two cases, we will denote the corresponding summands of $F, E$ by 
$F'_a, E'_a$, and $F''_a, E''_a$ respectively ($a \in \bZ$).
 Thus $F_a = F'_a, E_a = E'_a, F_{a-t} = F''_{a}, E_{a-t} = E''_{a}$ for any $a 
\in \bZ$. 
 
Thus we obtain a description of the categorical action of $\sll_{\bZ} \oplus 
\sll_{\bZ}$ on the Grothendieck group of $\mathcal{D}_t$:
\begin{theorem}\label{thrm:action_Deligne_cat_ss}
 Let $t \notin \bZ$ (so $\mathcal{D}_t$ is semisimple). The functors $F'_a, 
E'_a, F''_a, E''_a$ for $a \in \bZ$ define an action of the Lie algebra 
$\sll_{\bZ} \oplus \sll_{\bZ}$ on the complexified Grothendieck group $\bC 
\otimes_{\bZ} Gr(\mathcal{D}_t)$. 
 
 As a $\sll_{\bZ} \oplus \sll_{\bZ}$-module, $\bC 
\otimes_{\bZ}Gr(\mathcal{D}_t)$ is isomorphic to the tensor product 
$\mathfrak{F} \otimes \mathfrak{F}^{\vee}$, with $\left[ \mathbf{T}(\lam) 
\right]$ corresponding to $v_{\lam^{\bullet}} \otimes v_{\lam^{\circ}}$.
\end{theorem}

\section{Categorical \texorpdfstring{$\sll_{\bZ}$-actions}{actions} on 
\texorpdfstring{$\cV$}{V} for integer 
\texorpdfstring{$t$}{t}}\label{sec:actions_on_V}
Let $t \in \bZ$. Consider the $\sll_{\bZ}$-action on the category $\cV$, induced 
by the $\sll_{\bZ}$-action described in Section \ref{sec:skeletal_action}.

Recall that the category $\cV$, defined in Section \ref{sssec:abelian_Deligne}, 
is essentially ``glued'' from pieces of categories $Rep(\gl(m|n))$ ($m-n =t$) 
using the functors $\mathcal{F}_{m|n}: \cV \to Rep(\gl(m|n))$, $V_t \mapsto 
\bC^{m|n}$; these provide ``local'' equivalences, which allows us to reduce the 
study of the $\sll_{\bZ}$-action on $\cV$ to ``local'' studies in 
$Rep(\gl(m|n))$ for $m, n>>0$. The functors $\mathcal{F}_{m|n}$ are equivariant 
(a direct consequence of \ref{lem:equivar_skeletal_act}). 

The term ``local equivalences'' means the following: the categories $\cV$ and 
$Rep(\gl(m|n))$ have $\bZ_+$-filtrations, so that $$\mathcal{V}_t^k \cong 
Rep^k(\gl(m|n)) \cong Rep^k(\gl(m-1|n-1))$$ for $m, n >> k$. The subcategories 
$\mathcal{V}_t^k$ (resp. $Rep^k(\gl(m|n))$) are defined as full subcategories of 
$\cV$ (resp. $Rep(\gl(m|n))$) containing all the subquotients of finite direct 
sums of mixed tensor powers $V_t^{\otimes r} \otimes V_t^{* \otimes s}$, $r +s 
\leq k$ (resp. $\left(\bC^{m|n} \right)^{\otimes r} \otimes {\left(\bC^{m|n} 
\right)^*}^{\otimes s}$).

\begin{remark}
 These subcategories are not preserved by the functors $F, E$.
\end{remark}

Using this fact we will now compute the action of functors $F_a \in \End(\cV)$ 
on the standard objects $\mathbf{V}(\lam)$.

\subsection{Action on standard objects}\label{ssec:action_F_a_standard}
We now compute the action of $F_a$ on a standard object $\mathbf{V}(\lam) \in 
\cV$. The following proposition is valid for any $t \in \bC$.

\begin{proposition}\label{prop:action_standards}
 The object $F_a(\mathbf{V}(\lam))$ in $\cV$ is standardly-filtered:
 $$ 0 \to \mathbf{V}(\lam + \InnaD{\young(\bullet)}_a) \longrightarrow 
F_a(\mathbf{V}(\lam)) \longrightarrow \mathbf{V}(\lam - 
\square_{-\InnaC{(a+t)}}) \to 0. $$
\end{proposition}
\begin{proof}
If $t\notin \bZ$, then $\mathbf{V}(\lam) = \mathbf{T}(\lam)$, and the statement 
is true (see Section \ref{sec:categorical_sl_infty_actions_GL_t}), so we will 
assume that $t \in \bZ$.

First, recall from \cite[Corollary 6.2.2]{EHS} that 
\begin{equation}\label{eq:F_on_standard}
 0 \to \bigoplus_{\mu \in \lam + \InnaD{\young(\bullet)}} \mathbf{V}(\mu) 
\longrightarrow 
F(\mathbf{V}(\lam)) \longrightarrow \bigoplus_{\mu \in \lam - \square} 
\mathbf{V}(\mu) \to 0
\end{equation}

Therefore $F_a(\mathbf{V}(\lam))$ is filtered by standard objects for every $a 
\in \bC$. 

To find out which $\mathbf{V}(\mu)$ appear in $F_a(\mathbf{V}(\lam))$, let $m >> 
0$ and let $n:=m-t$. Denote by $W:= \bC^{m|n} $ the tautological representation 
of $\gl(m|n)$. Its even and odd parts will be denoted $W_0, W_1$.

Recall from Lemma \ref{lem:V_lam_highest_weight} that $\mathit{V}(\lam) := 
\mathcal{F}_{m|n}(\mathbf{V}(\lam))$ is a highest-weight $\gl(m|n)$-module, the 
image of the homomorphism from the Kac module $\mathit{K}(\lam)$ to 
$S^{\lam^{\bullet}}W \otimes S^{\lam^{\circ}}W^*$.
 
Now, the action of the endofunctor $F_a \in \End(Rep(\gl(m|n)))$ on the Kac 
module $\mathit{K}(\lam)$ can be described very explicitly, as was done in 
\cite{BLW}.
 
 Consider the decomposition of the Lie superalgebra $$\gl(m|n) = 
\mathfrak{g}_{-1} \oplus \mathfrak{g}_{0} \oplus \mathfrak{g}_1$$ where 
$\mathfrak{g}_{0}$ is even part of the superspace $\gl(m|n)$, and 
$\mathfrak{g}_{1} \cong V_0 \otimes (V_1)^*$ (so $\mathfrak{g}_{-1} \oplus 
\mathfrak{g}_1$ is the odd part of $\gl(m|n)$).
 Then 
 \begin{align*}
 &F(\mathit{K}(\lam)) = U\gl(m|n) \otimes_{U(\mathfrak{g}_{0} \oplus 
\mathfrak{g}_1)} \left( \mathit{L}_{\mathfrak{g}_0} (\lambda^{\bullet}) \otimes 
\mathit{L}_{\mathfrak{g}_1} (\lambda^{\circ}) \right) \otimes W =\\
 &= U\gl(m|n) \otimes_{U(\mathfrak{g}_{0} \oplus \mathfrak{g}_1)} \left( 
\mathit{L}_{\mathfrak{g}_0} (\lambda^{\bullet}) \otimes 
\mathit{L}_{\mathfrak{g}_1} (\lambda^{\circ})  \otimes W \right) 
 \end{align*}

Since $W$ has a $\mathfrak{g}_{0} \oplus \mathfrak{g}_1$-filtration with 
subquotients $W_0, W_1$, the $\gl(m|n)$-module $F(\mathit{K}(\lam))$ has a 
filtration by Kac modules with highest weights lying in the set $\{\vec{\lam} + 
\delta_j \rvert j = 1,\ldots,m+n\}$. 

In other words, $F(\mathit{K}(\lam))$ has a filtration by Kac modules 
$\mathit{K}(\mu)$, where $\mu \in \lam + \InnaD{\young(\bullet)} \sqcup \lam - 
\square$.

Thus $F_a(\mathit{K}(\lam))$ has a filtration by Kac modules. It was computed 
in 
\cite{BLW} that the Kac modules $\mathit{K}(\mu)$ appearing in 
$F_a(\mathit{K}(\lam))$ have $\vec{\mu} = \vec{\lam} + \delta_j$ where 
$\vec{\lam}_j - j +1 =  a$ if $1 \leq j \leq m$, and $- \vec{\lam}_j + j -2m =  
a+1$. Translating these into the language of bipartitions, we obtain: $\mu \in 
\lam + \InnaD{\young(\bullet)}_a \sqcup \lam - \square_{-\InnaC{(a+t)}}$, and 
hence 
$F_a(\mathit{V}(\lam))$ has a filtration by $\mathit{V}(\mu)$ for the same 
bipartitions $\mu$. The filtration \eqref{eq:F_on_standard} now implies the 
required result.

\end{proof}

\begin{remark}
 One can similarly show that the object $E_a(\mathbf{V}(\lam)) \subset V_t^* 
\otimes \mathbf{V}(\lam)$ is standardly-filtered:
 $$ 0 \to \mathbf{V}(\lam + \square_{-\InnaC{(a+t)}}) \longrightarrow 
E_a(\mathbf{V}(\lam)) \longrightarrow \mathbf{V}(\lam - 
\InnaD{\young(\bullet)}_{a}) \to 0 
$$
\end{remark}

\subsection{Grothendieck group of \texorpdfstring{$\cV$}{V} in case of integer 
\texorpdfstring{$t$}{t}}\label{ssec:groth_group}

Once again, let $t \in \bZ$, and recall the definition of the shifted 
representation $\mathfrak{F}^{\vee}_t$ of $\sll_{\bZ}$ (see Section 
\ref{sssec:Fock_duals}).

The tensor product $\mathfrak{F} \otimes \mathfrak{F}^{\vee}_t$ is again an 
$\sll_{\bZ}$-module, and Proposition \ref{prop:action_standards} implies:
\begin{theorem}\label{thrm:action_Deligne_cat_non_ss}
 Let $t \in \bZ$. The functors $F_a, E_a$ for $a \in \bZ$ define an action of 
the Lie algebra $\sll_{\bZ}$ on the complexified Grothendieck group $\bC 
\otimes_{\bZ} Gr(\cV)$. 
 
 We then have an isomorphism of $\sll_{\bZ}$-modules $$\bC \otimes_{\bZ} Gr(\cV) 
\stackrel{\sim{}}{\longrightarrow} \mathfrak{F} \otimes \mathfrak{F}^{\vee}_t, 
\; \; \left[ \mathbf{V}(\lam) \right] \mapsto v_{\lam^{\bullet}} \otimes 
v_{\lam^{\circ}}$$
\end{theorem}

\begin{remark}
 The Grothendieck group of the full subcategory $\mathcal{V}_t^k \subset \cV$ 
corresponds to the subspace $$ \varinjlim_{r ,\; s: \; r+s \leq k} 
\mathfrak{F}_{(r)}\otimes \mathfrak{F}^{\vee}_{(s)} \subset \mathfrak{F} \otimes 
\mathfrak{F}^{\vee}$$ (see Secion \ref{sssec:Fock_inv_lim} for definition of 
subspace $\mathfrak{F}_{(r)} \subset \mathfrak{F}$; the subspace 
$\mathfrak{F}^{\vee}_{(s)} \subset \mathfrak{F}^{\vee}$ is defined 
analoguously). 
 
 The functor $\mathcal{F}_{m|n}:\cV \to Rep(\gl(m|n))$ is not exact, so it does 
not correspond to a map between $\mathfrak{F} \otimes \mathfrak{F}^{\vee}$ and 
$\wedge^m \bC^{\bZ} \otimes \wedge^n (\bC^{\bZ})^*$. Yet the equivalence 
$\mathcal{F}_{m|n}:\mathcal{V}_t^k \longrightarrow Rep^k(\gl(m|n))$ for $m, 
n>>0$ can be seen as a categorical version of the statement $$\varinjlim_{r ,\; 
s: \; r+s \leq k} \mathfrak{F}_{(r)}\otimes \mathfrak{F}^{\vee}_{(s)} \cong 
\varinjlim_{r ,\; s: \; r+s \leq k} \wedge^m \bC^{\bZ}_{(r)} \otimes \wedge^n 
\bC^{ \bZ}_{*, (s)}$$
 for $m, n>>k$.

\end{remark}

\subsection{Action on tilting objects}\label{ssec:action_F_a_tilting}

In this section we give some results on the action of the functors $F_a$ on 
objects $\mathbf{T}(\lam)$ in $\cV$ for $t \in \bZ$. 
\begin{notation}
For any $t \in \bC$, denote by $\tilde{A}_a(t)$ the matrix whose $(\lam, \mu)$ 
entry is $1$ iff $\mu\in \lam + \InnaD{\young(\bullet)}_a \bigsqcup \lam - 
\square_{-\InnaC{(a+t)}}$, and zero otherwise.
\end{notation}

Let $a, t \in \bC$. Recall that $F_a$ are exact endofunctors of the category 
$\cV$ which preserve the subcategory $\mathcal{D}_t$, so each 
$F_a(\mathbf{T}(\lam))$ decomposes as a direct sum of $\mathbf{T}(\mu)$.

In general, the formula \eqref{eq:F_a_action} does not hold. The best 
approximation is given by the next proposition.
Let $\lam, \mu$ be two bipartitions. Denote by $A^{\lam}_{a, \mu}(t)$ the 
multiplicity of $\mathbf{T}(\mu)$ in $F_a(\mathbf{T}(\lam))$. We will denote by 
$A_a(t)$ the matrix whose entries are $A^{\lam}_{a, \mu}(t)$ (the entries are 
indexed by pairs of bipartitions).

Recall that in Section \ref{sec:categorical_sl_infty_actions_GL_t} we showed 
that for $t \notin \bZ$, $A^{\lam}_{a, \mu}(t)$ is either $1$ or $0$, and it 
equals $1$ iff $\mu\in \lam + \InnaD{\young(\bullet)}_a \bigsqcup \lam - 
\square_{-\InnaC{(a+t)}}$.

\begin{proposition}
 Let $t\in \bZ$. Then $$A_a(t) \, = \, D(t)\, \tilde{A}_a(t) \,D(t)^{-1}$$ where 
$D(t)$ is the matrix whose entries are $D^{\lam}_{\mu}(t)$ (cf. Section 
\ref{sec:mult}).
\end{proposition}
This agrees with the fact that $A_a(t) = \tilde{A}_a(t)$ when $t \notin \bZ$, 
since in this case $D(t) = \id$.
\begin{proof}
 Recall from Theorem \ref{thrm:equiv_multip} that in the Grothendieck ring of 
$\cV$, we have $$\left[\mathbf{T}(\lam) \right] = \sum_{\mu} 
D^{\lam}_{\mu}(t)\left[\mathbf{V}(\mu)\right]$$ Multiplying both sides by 
$[V_t]$, we have:
 
 $$\left[\mathbf{T}(\lam) \otimes V_t\right]= \sum_{\mu} 
D^{\lam}_{\mu}(t)\left[\mathbf{V}(\mu) \otimes V_t\right] = \sum_{\mu, \mu'} 
D^{\lam}_{\mu}(t) \tilde{A}^{\mu}_{\mu', a}(t) \left[\mathbf{V}(\mu')\right] = 
\sum_{\mu, \mu'} D^{\lam}_{\mu}(t) \tilde{A}^{\mu}_{\mu', a}(t) 
\left[\mathbf{T}(\mu')\right] $$
 
 Yet by definition $$\left[\mathbf{T}(\lam) \otimes V_t\right]= 
\sum_{\mu}A^{\lam}_{\mu, a}(t) \left[\mathbf{T}(\mu)\right] $$ 

 Hence $A_a(t) \, = \, D(t)\, \tilde{A}_a(t) \,D(t)^{-1}$.
\end{proof}

Although computing explicitly the matrix $A_a(t)$ is difficult, one can show 
that part of it is identical to the entries in $\tilde{A}_a(t)$:

\begin{corollary}
 For any $\lam$, we have:
 $$ F_a(\mathbf{T}(\lambda)) = \bigoplus_{\mu\in \lam + 
\InnaD{\young(\bullet)}_a } 
\mathbf{T}(\mu) \oplus  \bigoplus_{\mu } \mathbf{T}(\mu)^{\oplus {A}^{\lam}_{a, 
\mu} (t)}.$$
\end{corollary}

\begin{proof}
Denote by $\abs{\lam}:=\abs{\lam^{\bullet}}+ \abs{\lam^{\circ}}$ the total size 
of a bipartition.

 Define an order by size on the bipartitions (so that $\lam \geq \mu$ if 
$\abs{\lam} \geq  \abs{\mu}$), and write the matrix $D(t)$ in this ordered 
basis. Then $D(t)$ is an lower-triangular matrix (since $D^{\lam}_{\mu}(t) \neq 
0$ implies $\lam \geq \mu$) with $1$ on the diagonal. The matrix $D^{-1}(t)$ is 
then also lower-triangular with $1$ on the diagonal. 
 
Meanwhile, the matrix $\tilde{A}_a(t)$ has $1$ only in positions $(\lam, \mu)$ 
where $\abs{\lam} = \abs{\mu} \pm 1$. We claim that in the product $A_a(t) = \, 
D(t)\, \tilde{A}_a(t) \,D(t)^{-1}$ all the entries above the diagonal are equal 
to the corresponding entry in $\tilde{A}_a(t)$. 

Indeed, if $\abs{\lam} < \abs{\mu}$, then $$A^{\lam}_{a, \mu}(t) = \sum_{\nu, 
\nu'} D^{\lam}_{\nu}(t)\, \tilde{A}^{\nu}_{a, \nu'}(t) \,\left( D(t)^{-1} 
\right)^{\nu'}_{\mu}$$

The above arguments imply that the only summands which are not zero correspond 
to bipartitions $\nu, \nu'$ such that $\abs{\nu'} = \abs{\nu} \pm 1$, 
$\abs{\nu'} \geq \abs{\mu}$ and $\abs{\nu} \leq \abs{\lam}$; these conditions 
imply that $\abs{\nu} = \abs{\lam} = \abs{\mu}-1 = \abs{\nu'} -1$. Since we 
require that $D^{\lam}_{\nu}(t), \left( D(t)^{-1} \right)^{\nu'}_{\mu} \neq 0$, 
this implies that $\nu = \lam$, $\nu' = \mu$, and we are done.

\end{proof}
\begin{remark}
 Moreover, the entries of $A_a(t) = \, D(t)\, \tilde{A}_a(t) \,D(t)^{-1}$ below 
the diagonal are non-zero only if $\mu\subset \lam$ and $(\abs{\lam^{\bullet}}, 
\abs{\lam^{\circ}}) =(\abs{\mu^{\bullet}}, \abs{\mu^{\circ}}) -(i, i+1)$ for 
some $i\geq 0$.
\end{remark}

\section{Tensor product categorification}\label{sec:tens_prod_categorif}
In this section, we show that $\cV$ is the tensor product of the categorical 
$\gl(\infty)$-modules $Pol \otimes Pol_t^{\vee}$, where $Pol$ stands for the 
category of polynomial representations of $\gl(\infty)$, and $Pol_t^{\vee}$ is 
the same category but with a modified $\sll_{\bZ}$ categorical action.

\subsection{Category of polynomial representations}\label{ssec:cate_poly_rep}
The category $Pol$ has several equivalent descriptions, see \cite[Section 
5]{SS}, \cite{En}. 

\begin{definition}[Category of polynomial representations]
 The category $Pol$ is the full monoidal Karoubian additive subcategory of the 
category of $\gl(\infty)$-modules generated by the standard representation 
$V_{\infty}:=\bC^{\infty}$ of $\gl(\infty)$.
\end{definition}
Equivalently, this is the free Karoubian additive SM $\bC$-linear category on 
one generator, $V_{\infty}$. This category is equivalent to the category of 
Schur functors, see \cite{SS}.

This is a semisimple symmetric monoidal category, with simple objects $S^{\nu}$, 
$\nu \in \mathcal{P}$ parametrized by the set of all Young diagrams.

Clearly, this can be considered as a \InnaC{lower highest-weight category in 
the sense of Definition \ref{def:lower_hw_cat}, with 
subcategories $Pol^{(k)}$ generated by $S^{\nu}$ where $\abs{\nu} \leq k$; each 
of these is a (semisimple) highest-weight category with a finite poset of 
weights. }

On this category, we have a natural type A action (see \cite{HY, L}), given by 
the functors $F(M):=  V_{\infty} \otimes M$ and its adjoint $E := 
\iota^L\left(V_{\infty, *} \otimes M \right)$ (here $V_{\infty, *}$ is the 
restricted dual of $V_{\infty}$, $\iota: Pol \to Mod_{\gl(\infty)}$ is the 
inclusion functor, and $\iota^R$ its left adjoint). The natural transformation 
$\tau$ is just the symmetry morphism, as in Section \ref{sec:skeletal_action}, 
and $\rx$ is a ``limit'' version of the natural transformation described in 
Section \ref{sec:skeletal_action}. 

Namely, one can consider $Pol$ as a certain limit of categories of polynomial 
representations of algebraic groups $GL_n$ as $n \to \infty$ (see \cite{HY}, 
\cite{En}). Under this identification, any $M \in Pol$ corresponds to a sequence 
$(M_n)_{n \geq 1}$, where $M_n$ is a polynomial $GL_n$-representation. Then 
$$V_{\infty}\otimes M = (\bC^n \otimes M_n)_{n \geq 1}$$ and $\rx_M: 
V_{\infty}\otimes M \longrightarrow V_{\infty}\otimes M $ is defined as 
$(\rx_n:=\sum_{1\leq i, j \leq n} E_{i, j} \otimes E_{i, j})_n$. 

This is a categorical version of the isomorphism \eqref{eq:Fock_limit}.

\mbox{}

On the simple objects, the functors $F_a$, $E_a$ ($a \in \bZ$) act by
$$ F_a(S^{\nu}) = S^{\nu + \square_a}, \;\;\; E_a(S^{\nu}) = S^{\nu - 
\square_a}$$
(as usual, if one of the Young diagrams is not defined, the corresponding module 
is considered to be zero).

This action categorifies the Fock space representation of $\sll_{\bZ}$ described 
in Section \ref{ssec:Fock_space}:
$$Gr \left( Pol \right) \longrightarrow \mathfrak{F}, \;\;\; [S^{\nu}]  \to 
v_{\nu}$$

One can easily see that this is the unique categorification of the 
$\sll_{\bZ}$-module $\mathfrak{F}$: indeed, the weight spaces in $\mathfrak{F}$ 
being one-dimensional, its categorification has to be semisimple, with simple 
objects parametrized by all Young diagrams. It follows that this 
categorification is equivariantly equivalent to $Pol$.

\mbox{}

We also consider a twist of this action: namely, let $t \in \bZ$, and consider 
the endofunctors $F' = V_{\infty, *} \otimes (\cdot)$, $E' =  V_{\infty} \otimes 
(\cdot)$ on the category $Pol$, together with the usual symmetry morphism $\tau' 
= \tau$, and $\rx' \in \End(E')$ defined by $\rx'= - \rx -t$ (this induces the 
corresponding endomorphism of $F'$). Then $F' = \bigoplus_{a \in \bZ} F'_a$, $E' 
= \bigoplus_{a \in \bZ} E'_a$, where 
$$ F'_a(S^{\nu}) = S^{\nu - \square_{-\InnaC{(a+t)}}}, \;\;\; E'_a(S^{\nu}) = 
S^{\nu + \square_{-\InnaC{(a+t)}}}$$
This action categorifies the ``shifted dual'' Fock space representation 
$\mathfrak{F}^{\vee}_t$ of $\sll_{\bZ}$ described in Section 
\ref{ssec:Fock_space}.

By abuse of notation, we denote the category of polynomial representations with 
the usual $\sll_{\bZ}$-action by $Pol$, and the category of polynomial 
representations with the twisted $\sll_{\bZ}$-action by $Pol_t^{\vee}$.
\subsection{Tensor product categorification}\label{ssec:tens_prod_categorif}
\InnaC{Let $t\in \bZ$. We now consider the notion of tensor product 
categorification $Pol \otimes Pol_t^{\vee}$ in a sense similar to \cite[Remark 
3.6]{LW}.

\begin{definition}\label{def:lower_hw_cat}
A lower highest weight category $\C$ is an artinian abelian $\bC$-linear 
category $\C$ together with a poset $ (\Lambda, \leq )$ (poset of weights) 
and a filtration $\Lambda = \bigcup_{k \in \bZ_+} \Lambda^k$, 
 such that the following conditions hold:
 \begin{enumerate}
 \item The set $\Lambda$ is in bijection with the set of isomorphism classes of 
simple objects in $\C$.

 \item For each $\xi \in \Lambda$, the Serre subcategory 
$\C( \leq \xi)$ generated by simples $\{L(\lambda), \lambda \leq \xi\}$ 
contains a projective cover $\Delta(\xi)$ of $L(\xi)$, and an injective 
hull $\nabla(\xi)$ of $\xi$. The objects $\Delta(\xi)$, $\nabla(\xi)$ are 
called {\it standard} and {\it costandard} objects in $\C$.

\item There exists precisely one isomorphism class of indecomposable objects 
$T(\xi)$ in $\C$ which has $\Delta(\xi)$ as a submodule, $T(\xi)/\Delta(\xi)$ 
has a filtration with standard subquotients, and $T(\xi)$ also has a filtration 
with costandard subquotients.

Such objects $T(\xi)$ are the indecomposable {\it tilting} objects in $\C$.

\item Let $k \geq 0$, and let $\C^k$ be the full subcategory of $\C$ whose 
objects are subquotients of finite direct sums of objects $T(\xi)$, $\xi \in 
\Lambda^k$. Then each $\C^k$ is a highest weight category with poset 
$(\Lambda_k, \leq)$, simple objects $\{L(\xi), \xi \in 
\Lambda^k\}$, standard objects $\{\Delta(\xi), \xi \in 
\Lambda^k\}$, costandard objects $\{\nabla(\xi), \xi \in 
\Lambda^k\}$, and tilting objects $\{T(\xi), \xi \in 
\Lambda^k\}$. The category $\C^k$ also has enough projective and injective 
objects.

\item The subcategories $\C^k$ form a filtration on the category $\C$: $\C = 
\bigcup_{k \geq 0} \C^k$.

 \end{enumerate}

\end{definition}
\begin{remark}
 The main difference between a highest-weight category and a lower 
highest-weight category is the possible lack of projectives and injectives in 
$\C$. 
\end{remark}

\begin{definition}\label{def:categorical_tensor_product}
 An lower highest weight category $\C$ with an $\sll_{\bZ}$-action $(E, F, \rx, 
\tau)$ is a tensor product categorification of $Pol \otimes Pol_t^{\vee}$ if it 
satisfies the following conditions:
 \begin{enumerate}
\item\label{cond:1} Its isomorphism classes of simple objects are in bijection 
with $ \mathcal{P} \times \mathcal{P}$ (direct product of the sets of simples 
in 
$Pol$, $Pol_t^{\vee}$), and the lower highest weight structure on $\C$ is 
compatible with the partial order given by $$ \lam \leq \mu \;\;\; \text{ if } 
\;\;\; \omega_{\lam^{\bullet}} - \omega_{\lam^{\circ}} = \omega_{\mu^{\bullet}} 
- \omega_{\mu^{\circ}}, \;\;\; \;\; \omega_{\lam^{\bullet}} \geq 
\omega_{\mu^{\bullet}}$$
\item\label{cond:2} The $\bZ_+$-filtration on the set of weights of $\C$ given 
in the lower highest weight structure corresponds to the filtration 
$\{\lambda \in \mathcal{P} \times \mathcal{P} :\abs{\lambda} \leq k\}$.
  \item\label{cond:3} Consider the standard objects $\Delta(\lambda)$ in $\C$. 
Then for any $a \in \bZ$,
  $$ 0 \to \Delta(\lam + \InnaD{\young(\bullet)}_a) \longrightarrow 
F_a(\Delta(\lam)) 
\longrightarrow \Delta(\lam - \square_{-\InnaC{(a+t)}}) \to 0 $$
  and 
  $$ 0 \to \Delta(\lam + \square_{-\InnaC{(a+t)}}) \longrightarrow 
E_a(\Delta(\lam)) 
\longrightarrow \Delta(\lam - \InnaD{\young(\bullet)}_{a}) \to 0 $$
  and similarly for costandard objects.
 \end{enumerate}
\end{definition}

\begin{remark}
 The requirements \eqref{cond:1}, \eqref{cond:3} above are explicit 
interpretations of the requirements of \cite[Remark 3.6]{LW}, with the exception 
that \cite[Remark 3.6]{LW} requires $\C$ only to be standardly stratified, while 
we give a stronger requirement of lower highest weight category. We do not know 
if this can be weakened to require a standardly stratified.
 
 An additional condition present in the requirements of \cite[Remark 3.6]{LW} is 
that the associated graded category $\bigoplus_{\xi} \quotient{\C(\leq 
\xi)}{\C(< \xi)}$ be equivalent to $Pol \boxtimes Pol_t^{\vee}$. In our case, an 
immediate consequence of the fact that $\C$ is a lower highest weight category 
is that the associated graded is semisimple, and so it is automatically 
equivalent to $Pol \boxtimes Pol_t^{\vee}$ due to \eqref{cond:1}.
\end{remark}

We present below an easy result on categories $\C$ satisfying the conditions in 
Definition \ref{def:categorical_tensor_product}.}

\InnaC{
Let $\C$ be a category satisfying the conditions in Definition 
\ref{def:categorical_tensor_product}, and let $U \in \C$ be the simple object 
corresponding to weight $(\InnaC{\varnothing}, \InnaC{\varnothing})$. 

Denote by $\C^{tilt}$ the full subcategory of tilting objects in $\C$, and by 
$\C'$ the (strictly) full Karoubian additive subcategory of $\C$ generated by 
objects of the form $G_1 \ldots G_m U$, where $G_1, \ldots, G_m \in \{E_a, F_a 
\mid a \in \bZ\}$, $m \geq 0$ (we will call such objects {\it translations of 
$U$}). }
\begin{lemma}\label{lem:categ_act_tilting}
 The subcategories $\C'$ and $\C^{tilt}$ coincide.
 \end{lemma}
 \begin{proof}
 Recall that the partial order on the weights in $\C$ satisfies \ref{cond:1} in 
Definition \ref{def:categorical_tensor_product}, the simple object $U$ is both 
standard and constandard, hence tilting. Condition \ref{cond:3} in Definition 
\ref{def:categorical_tensor_product} implies that for any $a \in \bZ$, the 
functors $E_a, F_a$ preserve the subcategories of standardly-filtered objects 
and of constandardly-filtered objects, and thus preserve the subcategory of 
tilting objects. Hence translations of $U$ are tilting objects, and $\C'\subset 
\C^{tilt}$.
 
 Next, Condition  \ref{cond:3} implies that for any $\lambda \in \mathcal{P} 
\times \mathcal{P}$, the standard object $\Delta(\lambda)$ occurs as a subobject 
of some $T\in \C'$. The corresponding indecomposable direct summand $T(\lambda)$ 
of $T$ is then the indecomposable tilting object $T(\lambda)$, implying that 
$\C^{tilt}\subset \C'$.
 \end{proof}

We \InnaC{now} show that the category $\cV$ with the $\sll_{\bZ}$-action 
satisfies the definition of a tensor product $Pol \otimes Pol_t^{\vee}$, and 
then show that such a tensor product is unique.

\begin{theorem}\label{thrm:tens_prod_categorif}
 The category $\cV$ with the action of $\sll_{\bZ}$ is a tensor product 
categorification $Pol \otimes Pol_t^{\vee}$ in the sense of \InnaC{Definition 
\ref{def:categorical_tensor_product}.}
\end{theorem}
\begin{proof}

To prove Condition \eqref{cond:1}, we need to show that $[\mathbf{V}(\lam): 
\mathbf{L}(\mu)]\neq 0$ $\Longrightarrow$ $\lam \InnaC{\leq} \mu$ \InnaC{in 
the order given by the lower highest weight structure on $\mathcal{V}_t$}. 
Indeed,  
recall that by \cite[Lemma 8.3.2]{EHS}, $[\mathbf{V}(\lam): 
\mathbf{L}(\mu)]\neq 
0$ iff one can obtain $d'_{\lam}$ from $d'_{\mu}$ by moving finitely many 
crosses in the cap diagram of $d'_{\mu}$ from the left end of a cap to the 
right 
end of this cap. This means that $d'_{\lam}$, $d'_{\mu}$ have the same core, 
and 
that $\sum_{i \leq a} \lam^{\bullet}_i - i \geq \sum_{i \leq a} \mu^{\bullet}_i 
- i $ for \InnaC{any $a \geq 1$}. 

By Lemma \ref{lem:equal_cores_weights}, $core(d'_{\lam}) = core(d'_{\mu})$ 
implies $\omega_{\lam^{\bullet}} - \omega_{\lam^{\circ}} = 
\omega_{\mu^{\bullet}} - \omega_{\mu^{\circ}}$,  and \InnaC{$\sum_{i \leq a} 
\lam^{\bullet}_i - i \geq \sum_{i \leq a} \mu^{\bullet}_i 
- i $ implies
$\omega_{\lam^{\bullet}} \geq \omega_{\mu^{\bullet}}$. Hence $\lambda \leq 
\mu$.}

\InnaC{Finally, to prove Condition \eqref{cond:3}, recall that the statement on 
translation of standard objects by functors $E_a, F_a$ was proved in Proposition 
\ref{prop:action_standards}. To prove the analogous statement for translation of 
costandard objects, recall from \cite{EHS} that $\cV$ possesses a contravariant 
endofunctor $(\cdot)^{\vee}: \cV^{op} \to \cV$ which is exact, interchanges 
standard objects with constandard objects (thus preserving moxed tensor power of 
$V_t$, which are tilting), and which commutes with tensor products. Applying the 
functor $(\cdot)^{\vee}$ to the exact sequences in Proposition 
\ref{prop:action_standards}, we obtain the required results for costandard 
objects. }
\end{proof}
\InnaC{\begin{theorem}\label{thrm:uniqueness_tens_prod_categ}
 The tensor product $Pol \otimes Pol_t^{\vee}$ is unique in the following sense: 
consider a lower highest-weight category $\C$ with an $\sll_{\bZ}$-action $(F', 
E', \rx', \tau')$ satisfying \InnaC{conditions in Definition 
\ref{def:categorical_tensor_product}}, and such that the natural transformation 
$\tau'_{E'F'}: E'F' \to F'E'$ induced by $\tau'$ is an isomorphism.

 Then we have a strongly equivariant equivalence $\C \cong \cV$.
\end{theorem}}
\begin{remark}
 The author has been told by J. Brundan that the requirement on $\tau'_{E'F'}$ 
can be lifted (i.e. will hold automatically), due to Rouquier's 
``$K_0$-control'' theorem given in \cite{R}. This will be explained in detail 
elsewhere.
\end{remark}

\begin{proof}

Consider the oriented Brauer category $\mathcal{OB}$ (see for example 
\cite{BCNR}). This is the free symmetric monoidal $\bC$-linear category 
generated by a single object $X$ and its dual. The skeleton subcategory 
${\cD^{0}_t}$ of the Deligne category $\underline{Rep}(GL_t)$, containing the 
mixed tensor powers of $V_t$, is then the specialization of $\mathcal{OB}$ 
under the relation $\dim(X) = t$, where $$\dim(\InnaC{X}) := ev_\InnaC{X} \circ 
\sigma_{\InnaC{X}, \InnaC{X}^*} \circ coev_\InnaC{X} \in 
\End_{\mathcal{OB}}(\triv)$$

Given any category $\C$ with an $\sll_{\bZ}$-action $(F', E', \rx', \tau')$, we 
have a monoidal functor $\Psi$ from the oriented Brauer category $\mathcal{OB}$ 
to $\End(\C)$ taking $X$ to $F$, $X^*$ to $E$, $\sigma_{\InnaC{X}, \InnaC{X}}$ 
to $ev_\InnaC{X}$ and $coev_\InnaC{X}$ to $\eps, \eta$, the adjunction unit and 
counit of $E', F'$, and $\sigma_{\InnaC{X}, \InnaC{X}}$ to $\tau'$. To check 
that this is indeed a functor, we need to check that $E', F', \eps, \eta, 
\tau'$ satisfy the relations in \cite[Theorem 1.1]{BCNR}; this is indeed the 
case, due to the conditions on $E', F', \tau'$ in the definition of a 
$\sll_{\bZ}$ categorical action, \InnaC{together with the requirement that the 
natural transformation $\tau'_{E'F'}$ be invertible}.

Let $U$ be the simple object in $\C$ corresponding to $(\InnaC{\varnothing}, 
\InnaC{\varnothing}) \in \mathcal{P} \times \mathcal{P}$. We define a functor 
$\InnaC{\Phi}: \mathcal{OB} \to \C$ by setting $\InnaC{\Phi}(\cdot) := 
\Psi(\cdot) U$. Recall that $\mathcal{OB}$ has a type A action, defined just as 
in Definition \ref{def:skeletal_action_sl_2}: $F = \InnaC{X}\otimes (\cdot)$, 
$E = \InnaC{X}^* \otimes (\cdot)$, $\tau = \sigma_{\InnaC{X}, \InnaC{X}}$ is 
the symmetry morphism of $\InnaC{X}$ and the morphism $\rx_{\InnaC{X}^{\otimes 
r} \otimes \InnaC{X}^{* \otimes s}}: \InnaC{X}^{\otimes r+1} \otimes 
\InnaC{X}^{* \otimes s} \to \InnaC{X}^{\otimes r+1} \otimes \InnaC{X}^{* \otimes 
s}$ is given by
$$\rx_{\InnaC{X}^{\otimes r} \otimes \InnaC{X}^{* \otimes s}} = \sum_{i = 1}^r 
\sigma_{1, i+1} - \sum_{j=1}^s ev_{1, j}$$ where $\sigma_{1, i}$ denotes the 
symmetry morphism between the first and the $i$-th $\InnaC{X}$-factors, and 
$ev_{1, j}$ denotes the contraction of the first $\InnaC{X}$-factor and the 
$j$-th $\InnaC{X}^*$-factor. 

We claim that this functor is strongly equivariant, 
i.e. there exists a natural isomorphism $\zeta: F'\InnaC{\Phi}  \to  
\InnaC{\Phi} F$ such that 
\begin{itemize}                                       
\item The induced natural transformation $\zeta^{\vee}: \InnaC{\Phi} E \to 
E' \InnaC{\Phi}$ is an isomorphism,
\item $ \zeta \circ \rx'\InnaC{\Phi} = \InnaC{\Phi}\rx \circ \zeta: 
F'\InnaC{\Phi}  \to \InnaC{\Phi} F $,
\item $ \zeta F \circ  F'\zeta \circ \tau'\InnaC{\Phi} = \InnaC{\Phi}\tau \circ  
\zeta F \circ  F'\zeta: F'^2\InnaC{\Phi}  \to  \InnaC{\Phi} F^2$,
\end{itemize}

The natural isomorphism $\zeta$ is obvious, and the only non-trivial part of 
this statement is the fact that $\rx$ is taken to $\rx'$. Due to the degenerate 
Hecke algebra relations on $\rx', \tau'$ and the adjointness of $F', E'$, the 
natural transformation $$\rx'F'^r E'^s:F'^{r+1} E'^s \to F'^{r+1} E'^s $$ can be 
expressed in terms of $\tau'$, adjunction units and counits, and the natural 
transformation $$F'^r E'^s\rx':F'^{r} E'^s F \to F'^{r+1} E'^s F$$
This means that it is enough to check that $ \zeta \circ 
\rx'\InnaC{\Phi}\rvert_{ F'\InnaC{\Phi}\triv} \stackrel{?}{=} \InnaC{\Phi}\rx 
\circ \zeta \rvert_{\InnaC{\Phi} F \triv}$ where $\triv \in \mathcal{OB}$ is 
the unit object. This statement is clearly true: the morphisms $\rx_{\triv} : 
\InnaC{X} \to \InnaC{X}$ and $\rx'_U: F'U \to F'U$ are both zero (the latter 
can be seen from the action of $\sll_{\bZ}$ on $[U] \in Gr(\C)$).

Thus $\InnaC{\Phi}$ is a strongly equivariant functor. We now claim that it 
factors through the specialization ${\cD^{0}_t}$ of $\mathcal{OB}$. This means 
that the functor $\InnaC{\Phi}$ takes the morphism $\dim(\InnaC{X}) \in 
\End_{\mathcal{OB}}(\triv)$ to $t\id_U$. Indeed, since $U$ is a simple object, 
$\InnaC{\Phi}(\dim(\InnaC{X}))$ is a scalar multiple of $\id_U$. 

\InnaC{Let us call 
this scalar $t'$, and show that $t' = t$. Indeed, consider 
$\rx'^2\rvert_{E'(U)}: F'E'(U) \to F'E'(U)$. Then $$\zeta \circ \rx'^2 = 
\InnaC{\Phi} \rx^2 \circ \zeta: F'E' U \to F'E'U$$ while $$\rx^2 
\rvert_{\InnaC{X}^*}: \InnaC{X} \otimes \InnaC{X}^* \to \InnaC{X} \otimes 
\InnaC{X}^* = -\dim(\InnaC{X}) \rx.$$
Hence $\rx'^2 = -t'\rx' \in \End(F'E'(U))$. Now, $\rx'$ has two generalized
eigenvalues on $F'E'(U)$, which are $0$ and $-t$ (this follows from Condition 
\ref{cond:3}). This implies $t'=t$ and $\InnaC{\Phi}(\dim(\InnaC{X})) = 
t\id_U$. }

This proves the required result, and gives us a strongly equivariant 
$\bC$-linear functor $\InnaC{\Phi}: {\cD^{0}_t} \to \C$, and extends to a 
strongly equivariant additive $\bC$-linear functor $\InnaC{\Phi}: \cD_t \to \C$ 
from the additive envelope of ${\cD^{0}_t}$. 

\mbox{}

We now show that this functor extends to a strongly equivariant \InnaC{faithful} 
exact functor $\InnaC{\Phi}: \cV \to \C$. 

By the universal property of $\cV$ as the abelian envelope of $\cD_t$ (see 
\cite[Section 9]{EHS}), it is enough to show that the functor $\InnaC{\Phi}: 
\cD_t \to \C$ is pre-exact, i.e. that for every morphism $f$ in $\cD_t$ which is 
a monomorphism (resp. epimorphism) in $\cV$, the morphism $\InnaC{\Phi}(f)$ is 
again a monomorphism (resp. epimorphism) in $\C$.

The proof is analogous to \InnaC{arguments in \cite[Section 9]{EHS}}. We will 
consider the case of an epimorphism (the case of a monomorphism is similar). In 
\cite[Proposition 8.6.1]{EHS}, it was shown that given an epimorphism $f: A \to 
B$ in $\cD_t$, there exists a non-zero object $Z \in \InnaC{\cD_t}$ such that 
the epimorphism $f\otimes \id_Z: Z \otimes A \to Z \otimes B$ is split. 

This implies that $\InnaC{\Phi}(\id_Z \otimes f)$ is a split epimorphism as 
well. \InnaC{Write $Z =  \InnaC{e} \bigoplus_{i}\InnaC{V_t}^{\otimes r_i} 
\otimes \InnaC{V_t}^{* \otimes s_i}$ for some idempotent $e \in 
\End_{\cD}(\bigoplus_{i}\InnaC{V_t}^{\otimes r_i} \otimes \InnaC{V_t}^{* 
\otimes s_i})$. Then the idempotent $\Phi(e \otimes \id_A) \in \End_{\C}\left( 
\Phi(\bigoplus_{i}\InnaC{V_t}^{\otimes r_i} \otimes \InnaC{V_t}^{* \otimes 
s_i} \otimes A) \right)$ induces an idempotent $e'_A \in \bigoplus_{i} 
F'^{r_i}E'^{s_i} \Phi A$ (through the isomorphism 
$\oplus_i\zeta^{r_i}\zeta^{\vee s_i}$).

Similarly, $\Phi(e \otimes \id_B)$ induces an idempotent $e'_B \in 
\bigoplus_{i} 
F'^{r_i}E'^{s_i} \Phi B$. By the construction above, we have a commutative 
diagram

$$\begin{CD}
   \Psi(Z \otimes A) @>{\InnaC{\Phi}(\id_Z \otimes f)}>> \Psi(Z \otimes B)\\
   @VVV @VVV \\
   \InnaC{ e'_A} \bigoplus_i F'^{r_i} E'^{s_i} A @>{ \oplus_i F'^{r_i} 
E'^{s_i}\InnaC{\Phi}(f)}>>  \InnaC{e'_B} \bigoplus_i F'^{r_i} E'^{s_i} B
  \end{CD}
$$

where the vertical arrows are isomorphisms.} Thus $\oplus_i F'^{r_i} 
E'^{s_i}\InnaC{\Phi}(f)$ is a split epimorphism, i.e. $\oplus_i F'^{r_i} 
E'^{s_i} Coker \InnaC{\Phi}(f)=0$. Yet $F', E'$ are exact endofunctors of $\C$ 
which do not annihilate any simple object, as can be seen from the action of 
$\sll_{\bZ}$ on the Grothendieck group. This implies that $Coker 
\InnaC{\Phi}(f)=0$, and $\InnaC{\Phi}(f)$ is an epimorphism. 

Thus we obtained a strongly equivariant \InnaC{faithful} exact functor 
$\InnaC{\Phi}: \cV \to \C$, taking $\triv$ to the simple object $U$.

\InnaC{
 \begin{lemma}\label{auxlem:standard_to_standard}
  The functor $\Phi$ takes the standard object $\mathbf{V}(\lambda)$ to the 
standard object $\Delta(\lambda) \in \C$, and similarly for costandard objects.
 \end{lemma}
\begin{proof}
This follows directly from the facts that both $\cV$, $\C$ satisfy Condition 
\ref{cond:3}, $\Phi$ is exact, and
$$\Phi(\mathbf{V}(\varnothing, \varnothing) \cong \triv) = U \cong 
\Delta(\varnothing, \varnothing).$$ The argument for costandard obejcts is 
exactly the same.
\end{proof}
\begin{lemma}\label{auxlem:tilting_equiv}
 The functor $\Phi$ takes objects in $\cD_t$ to objects in $\C^{tilt}$, and 
induces an equivalence of categories $\Phi: \cD_t 
\stackrel{\sim}{\longrightarrow} \C^{tilt}$. In particular, 
$\mathbf{T}(\lambda)$ is sent to the tilting object $T(\lambda) \in \C$.
\end{lemma}
\begin{proof}
 \InnaC{By Lemma \ref{lem:categ_act_tilting}, $\C^{tilt}$ coincides with the 
Karoubian additive category generated by translations of $U$. Since $\cD_t$ is 
the Karoubian additive envelope of the full subcategory $\cD^0_t$ generated by 
mixed tensor powers of $V_t$, it is enough to check that the $\Phi$ induces an 
equivalence of categories between $\cD^0_t$ and the full subcategory $\C'^{0}$ 
of $\C$ whose objects are translations $G_1 \ldots G_m U$ where $G_1, \ldots, 
G_m \in \{E, F\}$. The fact the $\Phi(\cD^0_t) \subset \C'^{0}$ follows directly 
from the fact that $\Phi$ is equivariant, and we also conclude that $\Phi: 
\cD^0_t \to \C'^{0}$ is essentially surjective. Hence we only need to check that 
$\Phi$ is full (we already know it is faithful); that is, we need to prove that 
$$\dim \Hom_{\cD}(\mathbf T, \mathbf T') = \dim \Hom_{\C^{tilt}}(\Phi \mathbf 
T, \Phi \mathbf T') \;\; \forall \mathbf T, \mathbf T' \in \cD^0_t.$$
Lemma \ref{auxlem:standard_to_standard} together with the fact that $\Phi$ is 
exact now implies that $[\Phi \mathbf  T: \Delta(\lambda)] = [\mathbf  T: 
\mathbf{V}(\lambda)] $ for any $\mathbf T \in \cD^0_t$ and any $\lambda$. 

Let $\mathbf{V}(\lambda)^{\vee} \in \cV$ and $\nabla(\lambda) \in \C$ denote the 
respective costandard objects. Then for any $\mathbf T, \mathbf T' \in \cD^0_t$,
$$\dim \Hom_{\cD}(\mathbf T, \mathbf T') = \sum_{\lambda} [\mathbf  T: 
\mathbf{V}(\lambda)] [\mathbf  T': \mathbf{V}(\lambda)^{\vee}] = \sum_{\lambda} 
[\Phi \mathbf  T: \Delta(\lambda)] [\Phi \mathbf  T': \nabla(\lambda)]= \dim 
\Hom_{\C^{tilt}}(\Phi \mathbf T, \Phi \mathbf T').$$ This completes the proof of 
the lemma.}
\end{proof}

We now show that the functor $\Phi: \cV \to \C$ is fully faithful. For 
this, we consider the left adjoint functor $\InnaC{\Phi}_*: \C \to Ind-\cV$ and 
the counit of the adjunction $\eps: \Phi_* \Phi \to \id$. We need to check that 
it is an isomorphism. 
We use the presentation property of $\cD_t$ in $\cV$, as described in 
\cite{EHS}: for any object $M \in \cV$, there exist objects $T, T' \in \cD_t$ 
together with a surjective map $T \to M$ and an injective map $M \to T'$. In 
such a case, the statement that $\eps_M$ is an isomorphism clearly from the fact 
that $\eps_T, \eps_{T'}$ are isomorphisms, as shown in Lemma 
\ref{auxlem:tilting_equiv}.

It remains to check that the functor $\Phi: \cV \to \C$ is essentially 
surjective. Indeed, let $M \in \C$. 

\InnaC{Since $\C$ is a lower highest-weight category, there exists $k \geq 0$ 
such 
that $M \in \C^k$. Let $P, I \in \C^k$ be the projective cover and injective 
hull of $M$ in $\C^k$ respectively.
object an injective object $\C^k$, and a 
map $f:P \to I$ whose image is $M$. Hence it is enough to check that $P, I$ 
belong to the image of $\Phi$: since the functor $\Phi$ is full and exact, it 
will follow that the map $f:P \to I$ and $M=Im(f)$ lie in the image of $\Phi$ as 
well.

Now, the object $P$ is standardly-filtered, so it has a resolution by tilting 
objects in $\C^k$; similarly, $I$ is costandardly-filtered, so it has a 
resolution by tilting objects in $\C^k$. By Lemma \ref{auxlem:tilting_equiv}, 
the functor $\Phi$ induces an equivalence $\Phi: \cD_t 
\stackrel{\sim}{\longrightarrow} \C^{tilt}$; hence $P$, $I$ lie in the image of 
$\Phi$. This completes the proof of the theorem.}}

\end{proof}

When $t \notin \bZ$, the semisimple Deligne category $\mathcal{D}_t$ is clearly 
an (exterior) tensor product of the categorical $\sll_{\bZ}\times 
\InnaA{\sll_{\bZ}}$-module $Pol \boxtimes Pol_0^{\vee}$ in the sense of 
\InnaC{\cite[Remark 3.6]{LW}}. The uniqueness of such a tensor product 
categorification is straightforward, similarly to the uniqueness of the 
categorification of the $\sll_{\bZ}$-representation $\mathfrak{F}$.

\section{Future directions}

I. Losev suggested (\cite{L1}) two additional problems to be solved in the 
framework of this project, to be addressed in the future. 

\subsection{}
The first problem is to understand the multiplicities in the parabolic category 
$Par(\vec{t})$ (see \cite[Section 4]{Et}). Let us give a short description of 
such a category. Let $t_1, \ldots, t_n \in \bC$, and set $\vec{t}:= (t_1, 
\ldots, t_n)$, $t:= \sum_i t_i$. Consider the tensor categories 
$\mathcal{V}_{t_i}$ ($i=1,2, \ldots, n$) with generators $V_{t_i}$, $V^*_{t_i}$. 
Denote: $$\gl_{t_i} : = V_{t_i} \otimes 
V_{t_i}^*$$ The direct product $\mathcal{V}_{\vec{t}}:=\mathcal{V}_{t_1} 
\boxtimes \mathcal{V}_{t_2} \boxtimes \ldots \boxtimes\mathcal{V}_{t_n} $ is 
once again a tensor category; given any objects $A_{i_1} \in 
\mathcal{V}_{t_{i_1}}, \ldots, A_{i_r} \in \mathcal{V}_{t_{i_r}}$, we denote by 
$A_{i_1} \boxtimes \ldots \boxtimes A_{i_r}$ the object $$\triv \boxtimes \ldots 
\boxtimes \triv \boxtimes A_{j_1} \boxtimes \triv \boxtimes\ldots \boxtimes 
\triv \boxtimes A_{j_r} \boxtimes \triv \boxtimes\ldots \boxtimes \triv $$ in 
$\mathcal{V}_{\vec{t}}$, where $j_1 < j_2 < \ldots <j_r$ is a reordering of 
$(i_1, \ldots, i_r)$ in increasing order. We will consider objects 
$$\gl_{\vec{t}}:= \bigoplus_{1\leq i,j \leq n} V_{t_i}  \boxtimes 
V_{t_j}^*\;\;\; \mathfrak{n}_+ := \bigoplus_{i <j} V_{t_i}  \boxtimes 
V_{t_j}^*,\;\;\; \mathfrak{n}_- := \bigoplus_{i >j}   V_{t_i} \boxtimes 
V_{t_j}^*, \;\; \mathfrak{l}:= \bigoplus_{1\leq i \leq n} \gl_{t_i}$$ Thus 
$\gl_{\vec{t}} = \mathfrak{n}_+ \oplus \mathfrak{n}_- \oplus \mathfrak{l}$. This 
is a Lie algebra object, the image of $\gl_t \in \mathcal{V}_t$ under the exact 
SM functor $\mathcal{V}_t \longrightarrow \mathcal{V}_{\vec{t}}$ given by $$V_t 
\longmapsto V_{\vec{t}} := V_{t_1} \boxtimes\ldots\boxtimes V_{t_n}$$

By definition, any object of $\mathcal{V}_{\vec{t}}$ carries a natural action of 
$\mathfrak{l}$ (induced by the actions of $\gl_{t_i}$ on objects of 
$\mathcal{V}_{t_i}$ for each $i$).

The parabolic category $O^{\mathfrak{p}}_{\vec{t}}$ is defined to be the 
category of $U(\gl_{\vec{t}})$-modules in $Ind-\mathcal{V}_{\vec{t}}$ on which 
$\mathfrak{n}_+$ acts locally nilpotently, $\sll_{\vec{t}}:= \Ker( Tr: 
\mathfrak{l} \to \triv)$ acts naturally (via the embedding $\sll_{\vec{t}} 
\hookrightarrow \mathfrak{l}$, and the quotient $\mathfrak{z} = 
\quotient{\mathfrak{l}}{\sll_{\vec{t}}}$ acts semisimply.

Let $s_1, \ldots, s_n \in \bC$, and set $\vec{s}:=(s_1, \ldots, s_n)$. Let 
$\lam^1, \ldots, \lam^n$ be bipartitions, and consider the simple representation 
$\mathbf{L}_{s_i}(\lam^i)$ of $\gl_{t_i}$ in $\mathcal{V}_{t_i}$, given by the 
natural action of $\gl_{t_i}$, twisted by the character $s_i Tr: \gl_{t_i} \to 
\triv$. In this setting, one can define the parabolic Verma object 
$$M_{\vec{\lam}, \vec{s}} = \Ind^{\gl_{\vec{t}}}_{\mathfrak{l} \oplus 
\mathfrak{n}_+} \mathbf{L}_{s_1}(\lam^1) \boxtimes \ldots \boxtimes 
\mathbf{L}_{s_n}(\lam^i) $$ in $O^{\mathfrak{p}}_{\vec{t}}$. This object has a 
unique irreducible quotient $L_{\vec{\lam}, \vec{s}}$, and all simple objects of 
$O^{\mathfrak{p}}_{\vec{t}}$ are of this form (see \cite[Proposition 4.5]{Et}).

In \cite[Section 4]{Et}, P. Etingof suggested the problem of computing the 
Kazhdan-Lusztig coefficients in this category. I. Losev suggested that this 
should be done similarly to the computation of the (parabolic) super Kazhdan 
Lusztig coefficients in \cite{BLW}, by showing that the functors $F:= 
V_{\vec{t}}\otimes (\cdot)$, $E:= V_{\vec{t}}^* \otimes (\cdot)$ induce an 
$\sll_{\bZ}$-categorical action on the category $O^{\mathfrak{p}}_{\vec{t}}$. 
The Grothendieck group of $O^{\mathfrak{p}}_{\vec{t}}$ is then isomorphic, as an 
$\sll_{\bZ}$-module, to a tensor product of $r$ copies of $\mathfrak{F}$ and 
$\mathfrak{F}^{\vee}$ (with twisted actions). Computing the Kazhdan Lusztig 
coefficients would then be possible using techniques analogous to \cite{BLW}.
\subsection{}
The second problem is to construct a graded lift of $\cV$ in the sense of 
\cite{BLW}. This should allow to compute the graded Kazhdan-Lusztig 
multiplicities for $\cV$, similarly to \cite[Section 5]{BLW}.

\end{document}